\documentclass[11pt]{amsart}

\usepackage{geometry}

\usepackage{colortbl}
\usepackage{amsmath,amssymb}
\usepackage{mathrsfs}
\usepackage{graphicx}
\usepackage{amsmath}
\usepackage{dsfont}
\usepackage{multirow}
\usepackage{graphicx}
\usepackage{epstopdf}
\usepackage{color}
\usepackage{ulem}
\usepackage{hyperref}
\usepackage{enumerate}
\usepackage{caption}
\usepackage{subfig}
\usepackage{bm}
\usepackage{overpic}

\usepackage{tikz}
\usepackage{varioref}
\usepackage{hyperref}
\usepackage{cleveref}

\makeatletter

\newcommand{\Rmnum}[1]{\expandafter\@slowromancap\romannumeral #1@}
\makeatother

\def\abs#1{\left|#1\right|}

\crefformat{figure}{Fig.~#2#1#3}
\crefmultiformat{figure}{Figures~(#2#1#3)}{ and~(#2#1#3)}{, (#2#1#3)}{ and~(#2#1#3)}
\crefformat{table}{Table~#2#1#3}

\def\bx{\mbox{\boldmath $x$}}

\def\vv{\mbox{\boldmath $v$}}
\def\w{\mbox{\boldmath $w$}}

\newtheorem{thm}{Theorem}[section]

\newtheorem{algorithm}{Algorithm}[section]

\theoremstyle{definition}

\theoremstyle{remark}

\newtheorem{remark}{Remark}

\numberwithin{equation}{section}

\begin{document}

\title[]{Micromagnetics simulations and phase transitions of ferromagnetics with Dzyaloshinskii-Moriya interaction}

\author{Panchi Li}
\email{LiPanchi1994@163.com}
\address{School of Mathematical Sciences, Soochow University, Suzhou, 215006, China}
%\thanks{China Scholarships Council No. 202106920036}
%\author{Panchi Li\\2022.01.29}
%\address{***}
%%\curraddr{}
%\email{***}
%%\thanks{}
\author{Shuting Gu}
\email{gst1988@126.com}
\address{College of Big Data and Internet, Shenzhen Technology University, Shenzhen 518118, China}

\author{Jin Lan}
\email{lanjin@tju.edu.cn}
\address{Center for Joint Quantum Studies and Department of Physics, School of Science, Tianjin University, 92 Weijin Road, Tianjin 300072, China.}

\author{Jingrun Chen}
\email{jingrunchen@ustc.edu.cn}
\address{School of Mathematical Sciences, University of Science and Technology of China, Hefei, Anhui 230026, China \newline
\indent Suzhou Institute for Advanced Research, University of Science and Technology of China, Suzhou, Jiangsu 215123, China}

\author{Weiqing Ren}
\email{matrw@nus.edu.sg}
\address{Department of Mathematics, National University of Singapore, 119076, Singapore}

\author{Rui Du}
\email{durui@suda.edu.cn}
\address{School of Mathematical Sciences, Soochow University, Suzhou, 215006, China.\newline\indent Mathematical Center for Interdisciplinary Research, Soochow University, Suzhou, 215006, China.}

%
%%    author two information
%\author{Supervisors: Shuting Gu and Jingrun Chen}
%\address{***}
%%\curraddr{}
%\email{***}
%%\thanks{}
%

\graphicspath{{figures/}}

%    \subjclass is required.
%\subjclass[2020]{Primary }

\date{\today}

\dedicatory{}
%\keywords{
%
%}
%    Abstract is required.
%%========================================================================%%
\begin{abstract}
Magnetic skyrmions widely exist in a diverse range of magnetic systems, including chiral magnets with a non-centrosymmetric structure characterized by Dzyaloshinkii-Moriya interaction~(DMI).
In this study, we propose a generalized semi-implicit backward differentiation formula projection method, enabling the simulations of the Landau-Lifshitz~(LL) equation in chiral magnets in a typical time step-size of $1$ ps, markedly exceeding the limit subjected by existing numerical methods of typically $0.1$ ps.
Using micromagnetics simulations, we show that the LL equation with DMI reveals an intriguing dynamic instability in magnetization configurations as the damping varies. Both the isolated skyrmionium and skyrmionium clusters can be consequently produced using a simple initialization strategy and a specific damping parameter.
Assisted by the string method, the transition path between skyrmion and skyrmionium, along with the escape of a skyrmion from the skyrmion clusters, are then thoroughly examined.
The numerical methods developed in this work not only provide a reliable paradigm to investigate the skyrmion-based textures and their transition paths, but also facilitate the understandings for magnetization dynamics in complex magnetic systems.
%THe method allows a time step of 1 ps in micromagnetic simulations, much
% Magnetic skyrmions are observed in a diverse range of magnetic systems, including non-centrosymmetric bulk magnets with chirality.
% In this study, we propose a generalized semi-implicit backward differentiation formula projection approach to investigate the static magnetization configurations and dynamic stabilization of magnetic textures in a ferromagnet with DMI.
% Our method allows a time step of 1 ps in micromagnetics simulations, exceeding the temporal step-size required by existing numerical methods (typically 0.1 ps). Our investigation of stabilization controlled by the Landau-Lifshitz equation reveals an intriguing dynamic instability in magnetization configurations as the damping parameter magnitude varies. We demonstrate that both isolated skyrmionium and skyrmionium clusters can be spontaneously produced using a simple initialization strategy and a specific damping parameter. We employ the string method to examine the transition path between skyrmion-based textures, which is fundamental in the context of spintronics applications. We present complete transition pathways for systems containing isolated skyrmionium and isolated skyrmion, along with the escape of a skyrmion from the skyrmion cluster. This innovative approach provides a dependable strategy to investigate skyrmion-based textures and their transition paths, facilitating the comprehension of the magnetization dynamics crucial to spintronics applications.
\end{abstract}

\maketitle
%%========================================================================%%

%%========================================================================%%
\section{Introduction}
Magnetic skyrmions are whirling structures of magnetization that have been observed in diverse types of magnetic systems~\cite{Back_2020,WANG2022169905}. Due to the topological protection, skyrmions maintain the extremely stability under external perturbations, and are thus frequently treated as particle-like objects~\cite{Heinze2011,PhysRevLett.120.197203}.
In addition, the individual skyrmion possesses small size down to the nanometer range, and high mobility under electric currents~\cite{Fert2017,Sampaio2013}, thus is a promising candidate for future ultradense information storages and logic techniques~\cite{doi:10.1126/science.1166767,doi:10.1021/acs.chemrev.0c00297}.
%And since skyrmion size is down to nanometer range,

The formation of skyrmions is facilitated by the Dzyaloshinkii-Moriya interaction (DMI)~\cite{DZYALOSHINSKY1958241,PhysRev.120.91}, but for a range of moderate strength. When the DMI strength is weak, the ground state is the homogeneous ferromagnetic domain of uniform magnetization;
in another limit of strong DMI, the spin spiral state forms.
For the medium strength of DMI, skyrmion as well as skyrmionium, a similar texture yet with a trivial topology, are spontaneously engendered~\cite{Fert2017,arxiv.2210.00892,Siemens_2016,Zhang2018,PhysRevB.94.094420}.
%Despite trivial topology, skyrmionium is another particle-like texture that
These magnetic textures, depicting the inhomogeneous distribution of magnetizations, represent local minima of magnetic free energy in chiral magnets with nonzero DMI. Gradient descent methods~\cite{Jin2017} therefore can be applied to search these minima states. In dynamics, the Landau-Lifshitz~(LL) equation~\cite{LandauLifshitz1935,Gilbert1955} guides the evolution of magnetization.

Micromagnetics simulations is an important tool to study the magnetization dynamics, where the LL equation is solved numerically. There are vast literatures on numerical methods for the LL equation without DMI (see recent reviews \cite{cimrak2007survey,SIAMRev2006DevelopmentLLG} and the references therein).
During the past two decades, semi-implicit projection methods \cite{WANG2001357,An2021analysisCN,LI2021semi-implicitiLLG,XIE2020109104} and the tangent plane scheme \cite{ALOUGES2006femMMMAS,ALOUGES20121345} have been developed for micromagnetics simulations to achieve a suitable trade-off between efficiency and numerical stability. However, while the DMI is crucial in the generation and transition of skyrmion (and skyrmionium), the incorporation of DM field to micromagnetics simulations is rarely studied.
The prominent obstacle in the numerical modeling is the chiral boundary conditions (a nonhomogeneous Neumann boundary condition), due to the curling nature of DM field.
Beside inhomogeneity, the DM field also imposes stringent constraint for the temporal step-size, typically in $0.1$ ps in existing methods \cite{PFEILER2020106965,doi:10.1142/S0218202522500208}.

To address these challenges posed by the DMI, we develop a generalized semi-implicit backward differentiation formula (BDF) projection scheme with the second-order accuracy in both space and time to numerically solve the LL equation with DMI.
A time step-size of $1.0$ ps is permissible in our method, that substantially reduces the computational expenses by one order of magnitude. Dynamic instability of the LL dynamics is observed that distinct stable magnetization configurations stabilize contingent on the damping parameter, from the initialization $\mathbf{e}_3$. Diverse skyrmion textures such as isolated skyrmion, isolated skyrmionium and their clusters are therefore generated, and minimal energy path (MEP) between these textures are determined with the assistance of the string method.
In addition, the protection by the chiral boundary and the propelling by local magnetic field are also demonstrated.

The remaining sections of this paper are structured as follows. In Section \ref{sec:model&methods}, we present the second-order semi-implicit method for the LL equation, along with the harmonic map heat flow technique to investigate the magnetic free energy with DMI, and the string method  to locate the transition path.
Skyrmion-based textures under different circumstances are then generated in Section \ref{sec:micromagnetics-simulations}, and the phase transitions associated with these textures are visualized  in Section \ref{sec:application-string-method}.
% Subsequently, in Section \ref{sec:micromagnetics-simulations}, we generate numerous skyrmion textures. The phase transitions associated with these textures are then visualized in Section \ref{sec:application-string-method}.
The concluding remarks are given in Section \ref{sec:conclusion}.

\section{Models and numerical methods}
\label{sec:model&methods}

Magnetic skyrmions were initially discovered in 2009~\cite{doi:10.1126/science.1166767} and have since been extensively studied due to their unique features, including extraordinary metastability, and their small physical size~(usually below $\sim 100\;\mathrm{nm}$), which results in high mobility at low-current densities~\cite{doi:10.1021/acs.chemrev.0c00297,Jin2017,doi:10.1126/science.1195709}. On the atomic scale, DMI has the form
\begin{equation}
E_{DM} = -\sum_{\langle ij\rangle}\mathbf{d}_{ij}\cdot(\mathbf{m}_i\times\mathbf{m}_j),
\label{equ:atomic-DMI}
\end{equation}
where $\mathbf{d}_{ij}$ denotes the DMI vector between atomic indices $i$ and $j$, having a direction dependent on the system's type~\cite{PhysRevLett.128.167202}, $\mathbf{m}_i$ represents the atomic moment with a unit length, and the summation is over all atomic indices that have finite neighbor interactions, $\langle ij\rangle$~\cite{SkyrmionsAlbert2013,PhysRevB.88.184422}. The total spin Hamiltonian comprises the Heisenberg exchange interaction, the DMI, and the interaction with an applied magnetic field $\mathbf{H}$, represented as:
\begin{equation}
E = -J\sum_{\langle ij\rangle}\mathbf{m}_i\cdot\mathbf{m}_j - \sum_{\langle ij\rangle}\mathbf{d}_{ij}\cdot(\mathbf{m}_i\times\mathbf{m}_j) - \mu_0\sum_i\mathbf{H}\cdot\mathbf{m}_i,
\label{equ:Heisenberg-equation}
\end{equation}
where $J$ is the exchange coupling constant, and $\mu_0$ denotes the vacuum permeability. The Heisenberg model \eqref{equ:Heisenberg-equation} outlined in~\cite{Jin2017,PhysRevLett.121.197202} disregards the anisotropy term and the stray field.

\subsection{The continuum model}

In the continuum model, the magnetization is represented as a vector field that is dependent on the spatial variable, $\mathbf{M} = \mathbf{M}(\bx)$ where $\bx\in\Omega$. For non-centrosymmetric bulk materials that have chirality, the continuous equivalent of \eqref{equ:Heisenberg-equation} is given by:
\begin{multline}
  \mathcal{F}[\mathbf{M}] = \int_\Omega\frac{A}{M_s^2}|\nabla\mathbf{M}|^2\mathrm{d}\bx + \frac{D}{M_s^2}\int_{\Omega}(\nabla\times \mathbf{M})\cdot\mathbf{M}\mathrm{d}\bx \\- \mu_0\int_{\Omega}\mathbf{H}\cdot\mathbf{M} \mathrm{d}\bx + \frac{\mu_0}{2}\int_{\mathbb{R}^3}|\nabla U|^2 \mathrm{d}\bx + \int_{\Omega}\Phi\left(\frac{\mathbf{M}}{M_s}\right)\mathrm{d}\bx,
\label{equ:LL-energy}
\end{multline}
where $A$ is the exchange constant, $D$ is the DMI constant, and $\Omega\subset\mathds{R}^d \; (d = 1, 2,3)$ is the region occupied by the magnetic body. Below the Curie temperature, the saturation magnetization $M_s$ is a constant and it satisfies $|\mathbf{M}(\bx)| = M_s$. The fourth term is the dipolar energy that is defined by the Newtonian potential $N(\bx) = -\frac 1{4\pi}\frac 1{|\bx|}$ in the form
\begin{equation}
  U(\bx) = \int_{\Omega}\nabla N(\bx-\bx')\cdot\mathbf{M}(\bx')\mathrm{d}\bx'.
\end{equation}
Denote the stray field by $\mathbf{H}_s = -\nabla U(\bx)$, then the dipolar energy can be rewritten as $-\frac{\mu_0}{2}\int_{\Omega}\mathbf{H}_s\cdot\mathbf{M}\mathrm{d}\bx$. The anisotropy energy $\Phi(\mathbf{M}(\bx)/M_s): \Omega\rightarrow\mathds{R}^+$ is a smooth function. For a uniaxial ferromagnet with the easy-axis direction $\mathbf{e}_1 = (1,0,0)^T$, the anisotropy energy has the form $\Phi(\mathbf{M}(\bx)/M_s) = K_u(M_2^2 + M_3^2)/M_s^2 = K_u(M_s^2-M_1^2)/M_s^2$ with $K_u$ the anisotropy constant. A magnetic skyrmion induced by the DMI has a topology number (or skyrmion number), which is defined by~\cite{Heinze2011,Fert2017}
\begin{equation}
  Q = \frac 1{4\pi M_s^3}\int\mathbf{M}\cdot\left(\frac{\partial\mathbf{M}}{\partial x}\times\frac{\partial\mathbf{M}}{\partial y}\right)\mathrm{d}x\mathrm{d}y = \pm 1.
\end{equation}
In a skyrmion lattice, the topology number is proportional to the accumulated isolated skyrmion.

In the dynamic case, the magnetization $\mathbf{M} = \mathbf{M}(\bx, t)$ follows the phenomenological Landau-Lifshitz-Gilbert~(LLG) equation (an equivalent of the LL equation)
\begin{equation}
  \frac{\partial\mathbf{M}}{\partial t} = -\gamma\mathbf{M}\times\boldsymbol{\mathcal{H}} + \frac{\alpha}{M_s}\mathbf{M}\times\frac{\partial\mathbf{M}}{\partial t},
\label{equ:LLG-equation}
\end{equation}
where $\gamma$ is the gyromagnetic parameter, $\alpha$ is the dimensionless damping parameter, and $\boldsymbol{\mathcal{H}}$ is the effective field calculated by the variation of the energy functional~\eqref{equ:LL-energy},
\begin{equation}
  \boldsymbol{\mathcal{H}} = \frac{2A}{M_s^2}\Delta\mathbf{M} - \frac{2K_u}{M_s^2}(M_2\mathbf{e}_2 + M_3\mathbf{e}_3) + \mu_0\mathbf{H} + \mu_0\mathbf{H}_{\mathrm{s}} - \frac{2D}{M_s^2}\nabla\times\mathbf{M}.
\end{equation}
As a result of calculus of variations, the non-homogeneous Neumann boundary condition, also known as the chiral boundary, is derived
\begin{equation}\label{eqn:bc}
  \frac{\partial\mathbf{M}}{\partial{\boldsymbol{\nu}}}\Big|_{\partial\Omega} = -\frac{D}{2A}\mathbf{M}\times\boldsymbol{\nu}
\end{equation}
with $\boldsymbol{\nu}$ being the unit outward normal vector. This is important for the energy dissipation law of the magnetization dynamics stated in the following theorem.
\begin{thm}
	Let $\mathbf{M}\in L^{\infty}([0,T];[H^1(\bar{\Omega})]^3)\cap C^{1}([0,T];[C^1(\bar{\Omega})]^3)$ be the solution of \eqref{equ:LLG-equation}-\eqref{eqn:bc}, then the following energy dissipation law holds
	\begin{equation}
	\frac{\mathrm{d}\mathcal{F}[\mathbf{M}]}{\mathrm{d}t} \leq 0
	\end{equation}
	if the external magnetic field is independent of time.
\end{thm}
\begin{proof}
	For vectors $\vv, \w\in H^1(\Omega)$, it holds that $\nabla\cdot(\vv\times\w) = \w\cdot(\nabla\times\vv) - \vv\cdot(\nabla\times\w)$.
	Taking the volume integral and applying the divergence theorem, we have
	\begin{equation}
	\int_{\Omega}\nabla\cdot(\vv\times\w)\mathrm{d}\bx = \int_{\Omega}\w\cdot(\nabla\times\vv) - \vv\cdot(\nabla\times\w)\mathrm{d}\bx = \int_{\partial\Omega}(\vv\times\w)\cdot\boldsymbol{\nu}\mathrm{d}S.
	\end{equation}
	Then, we get
	\begin{equation}
	\frac{\mathrm{d}}{\mathrm{d}t}\int_{\Omega}\mathbf{M}\cdot(\nabla\times\mathbf{M})\mathrm{d}\bx = 2\int_{\Omega}\frac{\partial\mathbf{M}}{\partial t}\cdot(\nabla\times\mathbf{M})\mathrm{d}\bx - \int_{\partial\Omega}\Big(\mathbf{M}\times\frac{\partial\mathbf{M}}{\partial t}\Big)\cdot\boldsymbol{\nu}\mathrm{d}S.
	\end{equation}
	Due to the nonhomogeneous boundary condition, we have
	\begin{align}
	\int_{\Omega}\frac{\partial\mathbf{M}}{\partial t}\cdot\Delta\mathbf{M}\mathrm{d}\bx &= \int_{\partial\Omega}\frac{\partial\mathbf{M}}{\partial t}\cdot\nabla\mathbf{M}\cdot\boldsymbol{\nu}\mathrm{d}S - \frac 1{2}\frac{\mathrm{d}}{\mathrm{d}t}\int_{\Omega}|\nabla\mathbf{M}|^2\mathrm{d}\bx \nonumber\\
	&= -\frac{D}{2A}\int_{\partial\Omega}\frac{\partial\mathbf{M}}{\partial t}\cdot(\mathbf{M}\times\boldsymbol{\nu})\mathrm{d}S - \frac 1{2}\frac{\mathrm{d}}{\mathrm{d}t}\int_{\Omega}|\nabla\mathbf{M}|^2\mathrm{d}\bx \nonumber\\
	&= \frac{D}{2A}\int_{\partial\Omega}\Big(\mathbf{M}\times\frac{\partial\mathbf{M}}{\partial t}\Big)\cdot\boldsymbol{\nu}\mathrm{d}S - \frac 1{2}\frac{\mathrm{d}}{\mathrm{d}t}\int_{\Omega}|\nabla\mathbf{M}|^2\mathrm{d}\bx.
	\end{align}
	Taking inner product with \eqref{equ:LLG-equation} by $\gamma\boldsymbol{\mathcal{H}}-\frac{\alpha}{M_s}\frac{\partial\mathbf{M}}{\partial t}$, we arrive at
	\begin{align*}
	&0 \leq \frac{\alpha}{\gamma M_s}\int_{\Omega}\Big(\frac{\partial\mathbf{M}}{\partial t}\Big)^2\mathrm{d}\bx = \int_{\Omega}\frac{\partial\mathbf{M}}{\partial t}\cdot\boldsymbol{\mathcal{H}}\mathrm{d}\bx
	\\= &\frac{2A}{M_s^2}\int_{\Omega}\frac{\partial\mathbf{M}}{\partial t}\cdot\Delta\mathbf{M}\mathrm{d}\bx - \frac{2K_u}{M_s^2}\int_{\Omega}\frac{\partial\mathbf{M}}{\partial t}\cdot(M_2\mathbf{e}_2+M_3\mathbf{e}_3)\mathrm{d}\bx +\\ &\mu_0\int_{\Omega}\frac{\partial\mathbf{M}}{\partial t}\cdot\mathbf{H}_{\mathrm{s}}\mathrm{d}\bx + \mu_0\int_{\Omega}\frac{\partial\mathbf{M}}{\partial t}\cdot\mathbf{H}\mathrm{d}\bx - \frac{2D}{M_s^2}\int_{\Omega}\frac{\partial\mathbf{M}}{\partial t}\cdot(\nabla\times\mathbf{M})\mathrm{d}\bx\\
	= &\frac{D}{M_s^2}\int_{\partial\Omega}\Big(\mathbf{M}\times\frac{\partial\mathbf{M}}{\partial t}\Big)\cdot\boldsymbol{\nu}\mathrm{d}\bx - \frac{A}{M_s^2}\frac{\mathrm{d}}{\mathrm{d}t}\int_{\Omega}|\nabla\mathbf{M}|^2\mathrm{d}\bx - \frac{K_u}{M_s^2}\frac{\mathrm{d}}{\mathrm{d}t}\int_{\Omega}(M_2^2+M_3^2)\mathrm{d}\bx +\\
	&\mu_0\frac{\mathrm{d}}{\mathrm{d}t}\int_{\Omega}\mathbf{M}\cdot(\mathbf{H}+\mathbf{H}_{\mathrm{s}})\mathrm{d}\bx -\frac{D}{M_s^2}\frac{\mathrm{d}}{\mathrm{d}t}\int_{\Omega}\mathbf{M}\cdot(\nabla\times\mathbf{M})\mathrm{d}\bx - \frac{D}{M_s^2}\int_{\partial\Omega}\Big(\mathbf{M}\times\frac{\partial\mathbf{M}}{\partial t}\Big)\cdot\boldsymbol{\nu}\mathrm{d}\bx \\
	= &-\frac{\mathrm{d}\mathcal{F}}{\mathrm{d}t},
	\end{align*}
	where we use the fact that the external magnetic field is independent of time. This completes the proof.
\end{proof}

Skyrmion-based patterns induced by the DMI exhibit superior mobility in response to current fields. To account for the magnetic interactions with an external current, we include the spin transfer torque (STT) supplied by the spin-polarized current in the LLG model, as described in~\cite{PhysRevLett.93.127204}:
\begin{equation}
\frac{\partial\mathbf{M}}{\partial t} = -\gamma\mathbf{M}\times\boldsymbol{\mathcal{H}} + \frac{\alpha}{M_s}\mathbf{M}\times\frac{\partial\mathbf{M}}{\partial t} - \frac{b}{M_s^2}\mathbf{M}\times(\mathbf{M}\times(\mathbf{j}\cdot\nabla)\mathbf{M}) - \frac{b\xi}{M_s}\mathbf{M}\times(\mathbf{j}\cdot\nabla)\mathbf{M},
\label{equ:LLG-STT}
\end{equation}
Herein, $P$ denotes the polarization rate, $\mathbf{j}$ represents the current density vector, $b = P\mu_B/(eM_s(1+\xi^2))$, where $\mu_B$ is the Bohr magneton, $e$ is the elementary charge, and $\xi$ is a dimensionless parameter that characterizes the degree of non-adiabaticity.

Denote
\begin{equation*}
  \mathbf{\hat{H}} = - \frac{2K_u}{M_s^2}(M_2\mathbf{e}_2 + M_3\mathbf{e}_3) + \mu_0\mathbf{H} + \mu_0\mathbf{H}_{\mathrm{s}} + \frac{b}{\gamma M_s^2}\mathbf{M}\times(\mathbf{j}\cdot\nabla)\mathbf{M} + \frac{b\xi}{\gamma M_s}(\mathbf{j}\cdot\nabla)\mathbf{M},
\end{equation*}
then \eqref{equ:LLG-STT} can be rewritten as the LL form
\begin{multline}
  \frac{\partial\mathbf{M}}{\partial t} = -\frac{\gamma}{1+\alpha^2}\mathbf{M}\times\Big(\frac{2A}{M_s^2}\Delta\mathbf{M}-\frac{2D}{M_s^2}\nabla\times\mathbf{M} + \mathbf{\hat{H}}\Big) -\\ \frac{\gamma\alpha}{1+\alpha^2}\mathbf{M}\times\Big(\mathbf{M}\times\Big(\frac{2A}{M_s^2}\Delta\mathbf{M}-
\frac{2D}{M_s^2}\nabla\times\mathbf{M} + \mathbf{\hat{H}}\Big)\Big).
\end{multline}
Define the dimensionless variables $\mathbf{m} = \mathbf{M}/M_s$, $\mathbf{h}_{\mathrm{e}} = \mathbf{H}/M_s$, $\mathbf{h}_{\mathrm{s}} = \mathbf{H}_{\mathrm{s}}/M_s$ and spatial rescaling $\bx \rightarrow L\bx$ with $L$ being the length of ferromagnetic body. The dimensionless magnetic free energy $I[\mathbf{m}]$ satisfying $\mathcal{F}[\mathbf{M}] = \mu_0M_s^2I[\mathbf{m}]$ is
\begin{equation}
  I[\mathbf{m}] =\int_{\Omega}\frac{\epsilon}{2}|\nabla\mathbf{m}|^2 + \frac{\kappa}{2}(\nabla\times \mathbf{m})\cdot\mathbf{m} - \mathbf{h}_{\mathrm{e}}\cdot\mathbf{m} - \frac{1}{2}\mathbf{h}_{\mathrm{s}}\cdot\mathbf{m} + \Phi\left(\mathbf{m}\right)\mathrm{d}\bx
\end{equation}
with $\epsilon = 2A/(\mu_0M_s^2L^2)$ and $\kappa = 2D/(\mu_0M_s^2L)$. Meanwhile, we take the time rescaling $t \rightarrow (1+\alpha^2)(\mu_0\gamma M_s)^{-1}t$ and get the dimensionless LL equation
\begin{equation}
  \frac{\partial\mathbf{m}}{\partial t} = -\mathbf{m}\times\Big(\epsilon\Delta\mathbf{m}-\kappa\nabla\times\mathbf{m} + \mathbf{\hat{h}}\Big) - \alpha\mathbf{m}\times\Big(\mathbf{m}\times\Big(\epsilon\Delta\mathbf{m}-
\kappa\nabla\times\mathbf{m} + \mathbf{\hat{h}}\Big)\Big),
\end{equation}
where $\mathbf{\hat{h}} = - q(m_2\mathbf{e}_2 + m_3\mathbf{e}_3) + \mathbf{h}_{\mathrm{e}} + \mathbf{h}_{\mathrm{s}} + \frac{b}{\mu_0\gamma M_s^2}\mathbf{m}\times(\mathbf{j}\cdot\nabla)\mathbf{m} + \frac{b\xi}{\mu_0\gamma M_s^2}(\mathbf{j}\cdot\nabla)\mathbf{m}$ with $q = 2K_u/(\mu_0M_s^2)$. Note that $\mathbf{\hat{H}} = M_s\mathbf{\hat{h}}$.

In the absence of spin-polarized current, i.e. $\mathbf{j} = \mathbf{0}$,  we have the following LL equation with boundary and initial conditions
\begin{equation}\left\{\begin{aligned}
  \frac{\partial\mathbf{m}}{\partial t} &= -\mathbf{m}\times\mathbf{h} - \alpha\mathbf{m}\times\left(\mathbf{m}\times\mathbf{h}\right)\;\;\;  &\mathrm{in}~[0, T]\times\Omega,\\
  \frac{\partial\mathbf{m}}{\partial\boldsymbol{\nu}} &= -\kappa_b\mathbf{m}\times\boldsymbol{\nu}  &\mathrm{on}~[0, T]\times\partial\Omega,\\
  \mathbf{m}(0) &= \mathbf{m}^0\;\;\;\mathrm{with} \abs{\mathbf{m}^0} = 1 &\mathrm{in}~\{t=0\}\times\Omega,
\end{aligned}\right.
\end{equation}
where
\begin{equation}
  \mathbf{h} = -\frac{\delta I}{\delta\mathbf{m}} = \epsilon\Delta\mathbf{m}- \kappa\nabla\times\mathbf{m} - q(m_2\mathbf{e}_2 + m_3\mathbf{e}_3) + \mathbf{h}_{\mathrm{e}} + \mathbf{h}_s.
\end{equation}
Here $\kappa_b = DL/(2A)$ is proportional to $L$, implying the stiffness of the boundary.

A stable or metastable state, such as skyrmion and skyrmionium, given by the LL equation satisfies
\begin{equation}
  \mathbf{m} = c\mathbf{h}
\end{equation}
with $c$ being a constant for $|\mathbf{m}| = 1$ in a point-wise sense. Due to the energy dispassion of the LL equation, the convergent procedure should stop at the local minimizer of the energy functional \eqref{equ:LL-energy} with
\begin{equation}
  \left\{\begin{aligned}
    &\frac{\delta{I}}{\delta\mathbf{m}} = 0,\\
    &\text{s.t. } |\mathbf{m}| = 1.
  \end{aligned}\right.
\end{equation}
Gradient decent methods, such as the nonlinear conjugate gradient method~\cite{Jin2017}, therefore can be applied to search the minima of the magnetic free energy. Alternatively, the stable magnetization configuration can be obtained by simulating the dynamics driven by the harmonic map heat flow equation. In this case, the minimization problem is formulated as
\begin{equation}
  \inf\Big\{ I(\mathbf{m}) \;|\; \frac{\partial\mathbf{m}}{\partial\boldsymbol{\nu}} = -\kappa_b\mathbf{m}\times\boldsymbol{\nu} \text{ on } \partial\Omega, |\mathbf{m}(\bx)| = 1, \forall \bx \in \Omega \Big\}.
\end{equation}
Using the Lagrange multiplier method with $\hat\lambda$ being the Lagrange multiplier, we get
\begin{equation}
  L(\mathbf{m}, \hat{\lambda}) = I(\mathbf{m}) + \frac{\hat{\lambda}}{2}\int_{\Omega}((\mathbf{m})^2 - 1)\mathrm{d}\bx.
\end{equation}
At stationary points, it holds
\begin{align*}
  \frac{\delta L}{\delta\mathbf{m}} = -\mathbf{h} + \hat{\lambda}\mathbf{m} = 0,\\
  (\mathbf{m})^2 - 1 = 0.
\end{align*}
So we have $\hat{\lambda} = (\mathbf{m}, \mathbf{h})$. Therefore, the harmonic map heat flow reads as
\begin{equation}
  \frac{\partial\mathbf{m}}{\partial t} = -\frac{\delta L}{\delta\mathbf{m}} = \mathbf{h} -(\mathbf{m},\mathbf{h})\mathbf{m}
\end{equation}
subject to the constraint $|\mathbf{m}| = 1$ and the nonhomogeneous boundary condition. Thus, we also consider the harmonic map heat flow system
\begin{equation}
\left\{\begin{aligned}
  &\frac{\partial\mathbf{m}}{\partial t} = -\mathbf{m}\times(\mathbf{m}\times\mathbf{h}),\;\;\;  &\mathrm{in}~[0, T]\times\Omega,\\
  &\frac{\partial\mathbf{m}}{\partial\boldsymbol{\nu}} = -\kappa_b\mathbf{m}\times\boldsymbol{\nu},\;\;\;  &\mathrm{on}~[0, T]\times\partial\Omega,\\
  &\mathbf{m}(0) = \mathbf{m}^0,\;\;\; &\mathrm{in}~\{t=0\}\times\Omega,
\end{aligned}\right.
  \label{equ:damping-equation}
\end{equation}
where $|\mathbf{m}^0| = 1$.

\subsection{Numerical methods}

FeGe is a representative chiral ferromagnet with strong spin-orbit coupling. Here we list its physical parameters in \cref{tab:physical-parmeters}.
\begin{table}[htbp]
  \centering
\caption{Physical parameters of FeGe ($L = 80\;\mathrm{nm}$).}
\begin{tabular}{||c|c|c||}
\hline
Parameter & Value (unit) & Dimensionless quantity\\
\hline
$K_u$ & $0\;(\mathrm{J}/\mathrm{m}^3)$ & $q = 0 $ \\
\hline
$A$ & $8.78\times10^{-12}\;(\mathrm{J}/\mathrm{m})$ & $\epsilon \approx1.48\times10^{-2}$\\
\hline
$M_s$ & $3.84\times10^5\;(A/\mathrm{m})$ &\\
\hline
$D$ & $1.58\times10^{-3}\;(\mathrm{J}/\mathrm{m}^2)$ & $\kappa \approx0.21$, $\kappa_d \approx7.20$\\
\hline
\end{tabular}
\label{tab:physical-parmeters}
\end{table}
It is clear that $\kappa$ and $\kappa_d$ are two leading parameters due to the DMI. The first parameter results a strong curl field in the LL equation and the second one leads to a stiff boundary condition, both of which will be examined later.

In the framework of finite difference method, we construct unknowns on half grid points as $\mathbf{m}(x_i,y_j,z_k) = \mathbf{m}((i-\frac 1{2})h_x, (j-\frac 1{2})h_y, (k-\frac 1{2})h_z)$. Here $h_x = 1/nx$, $h_y = 1/ny$, $h_z = 1/nz$ and $h = h_x = h_y = h_z$ holds for uniform spatial meshes. The indexes $i,j,k$ are valued with $i = 1,\cdots,nx$, $j = 1,\cdots,ny$ and $k = 1,\cdots, nz$. For the sake of clarity, the approximations of the boundary condition and operator $\nabla$ on $x$-direction are depicted below.
\begin{figure}[h]
  \centering
\includegraphics[width = 5.in]{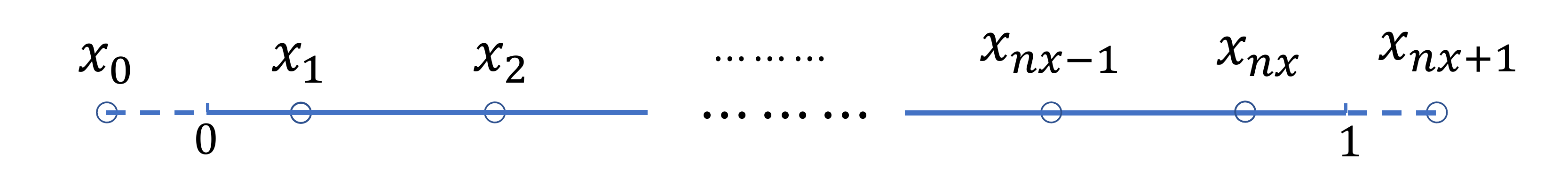}
    \caption{Grids along the $x$-direction with two ghost points $x_0$ and $x_{nx+1}$.}
    \label{fig:spatial-mesh}
\end{figure}

Let $\boldsymbol{\nu} = (1,0,0)^T$, the boundary condition at $x = 1$ is discretized as
\begin{equation*}
  \frac{\mathbf{m}(x_{nx+1},y_j,z_k) - \mathbf{m}(x_{nx},y_j,z_k)}{h_x} = -\kappa_b\frac{\mathbf{m}(x_{nx+1},y_j,z_k) + \mathbf{m}(x_{nx},y_j,z_k)}{2}\times\boldsymbol{\nu},
\end{equation*}
i.e.
\begin{equation*}
  \left\{\begin{aligned}
      &m_1(x_{nx+1},y_j,z_k) = m_1(x_{nx},y_j,z_k),\\
      &m_2(x_{nx+1},y_j,z_k) = \frac{1-k_{bx}^2}{1+k_{bx}^2}m_2(x_{nx},y_j,z_k) - \frac{2k_{bx}}{1+k_{bx}^2}m_3(x_{nx},y_j,z_k),\\
      &m_3(x_{nx+1},y_j,z_k) = \frac{2k_{bx}}{1+k_{bx}^2}m_2(x_{nx},y_j,z_k) + \frac{1-k_{bx}^2}{1+k_{bx}^2}m_3(x_{nx},y_j,z_k),
    \end{aligned}\right.
\label{equ:rboundary}
\end{equation*}
where $k_{bx} = \kappa_bh_x/2$. Similarly, discretization of the boundary condition at $x = 0$ yields
\begin{equation*}
    \left\{\begin{aligned}
      &m_1(x_0,y_j,z_k) = m_1(x_1,y_j,z_k),\\
      &m_2(x_0,y_j,z_k) = \frac{1-k_{bx}^2}{1+k_{bx}^2}m_2(x_1,y_j,z_k) + \frac{2k_{bx}}{1+k_{bx}^2}m_3(x_1,y_j,z_k),\\
      &m_3(x_0,y_j,z_k) = -\frac{2k_{bx}}{1+k_{bx}^2}m_2(x_1,y_j,z_k) + \frac{1-k_{bx}^2}{1+k_{bx}^2}m_3(x_1,y_j,z_k).
    \end{aligned}\right.
\label{equ:lboundary}
\end{equation*}
Boundary conditions along $y$ and $z$ directions are discretized in a similar way.% It is obvious that $|\mathbf{m}| = 1$ is maintained at all the ghost points.

The operator $\nabla_x$ is discretized as
\begin{align*}
  &\nabla_x\mathbf{m}(x_1,y_j,z_k) \approx \frac{\mathbf{m}(x_1,y_j,z_k) - \mathbf{m}(x_0,y_j,z_k)}{h_x},\\
  &\nabla_x\mathbf{m}(x_i,y_j,z_k) \approx \frac{\mathbf{m}(x_{i+1},y_j,z_k) - \mathbf{m}(x_{i-1},y_j,z_k)}{2h_x},\\
  &\nabla_x\mathbf{m}(x_{nx},y_j,z_k) \approx \frac{\mathbf{m}(x_{nx+1},y_j,z_k) - \mathbf{m}(x_{nx},y_j,z_k)}{h_x},
\end{align*}
where $i = 2,\cdots,nx-1$, $j = 1, \cdots, ny$ and $k = 1, \cdots, nz$. $\nabla_y$ and $\nabla_z$ are discretized similarly.
\begin{remark}
In thin films, the DMI arises from robust spin-orbit couplings at the edges, whereby magnetization is absent beyond the sample. In addition, the exchange interactions are non-symmetric and intrinsically directional at the boundaries. To numerically capture such features, the introduction of ghost points outside the material is commonly employed to approximate spatial derivatives within the vicinity of the boundaries. The resulting discretization of the $\nabla$ operator is achieved as described above.
\end{remark}
The Laplacian operator is descritized as
\begin{align*}
  \Delta\mathbf{m}(x_i,y_j,z_k) \approx &\frac{\mathbf{m}(x_{i-1},y_j,z_k) - 2\mathbf{m}(x_i,y_j,z_k) + \mathbf{m}(x_{i+1},y_j,z_k)}{h_x^2} + \\ &\frac{\mathbf{m}(x_i,y_{j-1},z_k) - 2\mathbf{m}(x_i,y_j,z_k) + \mathbf{m}(x_i,y_{j+1},z_k)}{h_y^2} + \\ &\frac{\mathbf{m}(x_i,y_j,z_{k-1}) - 2\mathbf{m}(x_i,y_j,z_k) + \mathbf{m}(x_i,y_j,z_{k+1})}{h_z^2}.
\end{align*}

Regarding time-stepping, the standard second-order backward differentiation formula (BDF2) is employed
\begin{multline}
  \frac{3\mathbf{m}^{n+1} - 4\mathbf{m}^n + \mathbf{m}^{n-1}}{2k} = -\mathbf{m}^{n+1}\times(\epsilon\Delta\mathbf{m}^{n+1} - \kappa\nabla\times\mathbf{m}^{n+1} + \mathbf{\hat{h}}^{n+1})\nonumber\\
      -\alpha\mathbf{m}^{n+1}\times(\mathbf{m}^{n+1}\times(\epsilon\Delta\mathbf{m}^{n+1} - \kappa\nabla\times\mathbf{m}^{n+1} + \mathbf{\hat{h}}^{n+1})).
\end{multline}
The prevalent feature of this approach involves the utilization of an implicit methodology, thereby necessitating nonlinear solvers at each time step. Notably, a semi-implicit scheme incorporating a projection step has been developed with the motivation of ensuring maintenance of $|\mathbf{m}| = 1$~\cite{XIE2020109104}. This scheme treats the DMI term as implicit, primarily due to its dominance in the effective field. The ensuing semi-implicit BDF2 projection scheme exhibits the following features. Given its second-order time accuracy, simulations in micromagnetics employing this scheme may adopt a step-size $\Delta t = 1\;\mathrm{ps}$. Conversely, earlier approaches were limited in their ability to utilize sub-picosecond time step-sizes.
\begin{algorithm}
  \label{alg:solving-LL}
   Set $\mathbf{\hat{m}}^{n+1} = 2\mathbf{m}^{n} - \mathbf{m}^{n-1}$ and $\mathbf{\tilde{h}}^{n+1} = 2\mathbf{\hat{h}}^{n} - \mathbf{\hat{h}}^{n-1}$.
  \begin{enumerate}[\rm(i)]
  \item Compute $\mathbf{\tilde{m}}^{n+1}$ such that
    \begin{align}\label{equ:solving-in-iteration}
      \frac{3\mathbf{\tilde{m}}^{n+1} - 4\mathbf{m}^n + \mathbf{m}^{n-1}}{2\Delta t} = -\mathbf{\hat{m}}^{n+1}\times(\epsilon\Delta\mathbf{\tilde{m}}^{n+1} - \kappa\nabla\times\mathbf{\tilde{m}}^{n+1} + \mathbf{\tilde{h}}^{n+1})\nonumber\\
      -\alpha\mathbf{\hat{m}}^{n+1}\times(\mathbf{\hat{m}}^{n+1}\times(\epsilon\Delta\mathbf{\tilde{m}}^{n+1} - \kappa\nabla\times\mathbf{\tilde{m}}^{n+1} + \mathbf{\tilde{h}}^{n+1}))
    \end{align}
  \item Projection onto $\mathcal{S}^2$:
    \begin{equation}
      \mathbf{m}^{n+1} = \frac 1{|\mathbf{\tilde{m}}^{n+1}|}\mathbf{\tilde{m}}^{n+1}
      \label{equ:projection}
    \end{equation}
\end{enumerate}
\end{algorithm}

For the harmonic map heat flow, a similar algorithm is proposed.
\begin{algorithm}
  \label{alg:solving-GF}
  Set $\mathbf{\hat{m}}^{n+1} = 2\mathbf{m}^{n} - \mathbf{m}^{n-1}$ and $\mathbf{\tilde{h}}^{n+1} = 2\mathbf{\hat{h}}^{n} - \mathbf{\hat{h}}^{n-1}$.
  \begin{enumerate}[\rm(i)]
    \item Compute $\mathbf{\tilde{m}}^{n+1}$ such that
    \begin{equation*}
      \frac{3\mathbf{\tilde{m}}^{n+1} - 4\mathbf{m}^n + \mathbf{m}^{n-1}}{2\Delta t} =
      -\mathbf{\hat{m}}^{n+1}\times(\mathbf{\hat{m}}^{n+1}\times(\epsilon\Delta\mathbf{\tilde{m}}^{n+1} - \kappa\nabla\times\mathbf{\tilde{m}}^{n+1} + \mathbf{\tilde{h}}^{n+1})).
    \end{equation*}
    \item Projection onto $\mathcal{S}^2$:
    \begin{equation*}
      \mathbf{m}^{n+1} = \frac 1{|\mathbf{\tilde{m}}^{n+1}|}\mathbf{\tilde{m}}^{n+1}.
    \end{equation*}
  \end{enumerate}
\end{algorithm}

The semi-implicit BDF1 approach is employed as a precursor to the BDF2 scheme, and allows for the calculation of $\mathbf{m}^1$. This initial step does not impart any alteration to the overall second-order accuracy exhibited by the numerical scheme. Additionally, it is worth mentioning that if the relative change in energy between two consecutive time steps is less than $1.0\times10^{-9}$ in the simulation, a steady state is considered to have been attained.

In order to search the minimum energy transition paths of skyrmion-based magnetic textures, here we further introduce the string method~\cite{PhysRevB.66.052301,doi:10.1063/1.2720838}. By definition, a curve $\boldsymbol{\gamma}$ connecting two local minima satisfies
\begin{equation}
  (\nabla I)^{\perp}(\boldsymbol{\gamma}) = 0,
\end{equation}
where $(\nabla I)^{\perp}$ is the component of $\nabla I$ normal to $\boldsymbol{\gamma}$. Then $\boldsymbol{\gamma}:=\{\varphi(a), a\in[0,1]\}$ defines a MEP from one local minima to the other. After an initial parametrization of the curve is picked and usually the equal arc-length parametrization is used, the curve evolve to the MEP following the equation
\begin{equation}
  \varphi_t = -(\nabla I(\varphi))^{\perp} + \lambda\tau,
  \label{equ:Lag-Grad-flow}
\end{equation}
where $(\nabla I(\varphi))^{\perp} = \nabla I(\varphi) - (\nabla I(\varphi), \tau)\tau$, $\tau$ is the unit tangent vector along $\varphi$ with $\tau = \varphi_{a}/\abs{\varphi_{a}}$, and $\lambda$ is the Lagrange multiplier uniquely determined by the choice of parametrization. Let $\bar{\lambda} = \lambda + (\nabla I(\varphi), \tau)$, then the \eqref{equ:Lag-Grad-flow} can be rewritten as
\begin{equation}
  \varphi_t = -\nabla I(\varphi) + \bar{\lambda}\tau.
  \label{equ:Lag-Grad-flow2}
\end{equation}
In order to find the MEP, the time-splitting method is applied to solve \eqref{equ:Lag-Grad-flow2}. Details are given in Algorithm \ref{alg:string-method}. A convergent string means that all the images satisfy
\begin{equation}
  (\nabla I(\varphi_i))^{\perp} = 0.
\end{equation}
In our simulations, this is replaced by the stopping criterion ($TOL = 1.0\text{e-}06$) as
\begin{equation}
  \max_i||I(\varphi_i, t^n) - I(\varphi_i, t^{n+1})||_{\infty} \leq TOL.
  \label{equ:stationary_criterion}
\end{equation}

\begin{algorithm}
  \label{alg:string-method}
Choose an initial string $\boldsymbol{\gamma}^0$ with inclusion of $N+1$ images such that $\boldsymbol{\gamma}^0:=\{\varphi_i = \varphi(a_i), a_i\in[0,1], i = 0, \cdots, N\}$.
\begin{itemize}
\item{Step 1:} Evolve the images on the string following the gradient flow
  \begin{equation}
    \partial_t\varphi_{i} = -\nabla I(\varphi_i).
  \end{equation}
  From current images $\{\varphi_i^n\}$, $\{\varphi_i^*\}$ are obtained by one time stepping. Here an image is denoted by the magnetization configuration $\mathbf{m}(\bx), \bx\in\Omega$, and the harmonic map heat flow \eqref{equ:damping-equation} is solved by Algorithm \ref{alg:solving-GF}.
  \item{Step 2:} Compute the parametrization $\{a_i^*\}$ by
      \begin{equation*}
        s_0 =0, s_i = s_{i-1}+\abs{\varphi_i^* - \varphi_{i-1}^*}, i = 1,\cdots, N,
      \end{equation*}
      and then the updated mesh $\{a_i^*\}$ is normalized by $a_i^* = s_i/s_N$.
\item{Step 3:} Parametrization of the string by equal arc-length and projection. The images $\{\hat\varphi_i^{n+1}\}$ are obtained by cubic spline interpolation at uniform grid points $\{a_i = i/N\}$, and the new images $\{\varphi_i^{n+1}\}$ are obtained after the projection step.
\item{Step 4:} Go back to Step 1 and iterate until convergence.
\end{itemize}
\end{algorithm}

\begin{remark}
The approach to the string method utilized in this present work differs from the standard version outlined in~\cite{doi:10.1063/1.2720838,doi:10.1063/1.1536737}. This variance essentially stems from the novel conservation requirement on the length of magnetization. Accordingly, Step 3 in the method is crafted to incorporate a projection step. Furthermore, treating images as tensors characterized by $\varphi_i\in\mathds{R}^{3*nx*ny*nz}$ makes it possible to consider the existence of a reversible mapping between $\varphi$ and $\mathbf{m}$, which is captured by mappings $\mathcal{L}: \varphi\rightarrow \mathbf{m}$ and $\mathcal{L}^{-1}: \mathbf{m}\rightarrow \varphi$ in the implemented technique.
\end{remark}

\section{Micromagnetics simulations}
\label{sec:micromagnetics-simulations}
This section initiates with the utilization of Algorithm \ref{alg:solving-LL} and Algorithm \ref{alg:solving-GF} in generating stable magnetic textures. The simulation carried out considers a FeGe sample having a consistent spatial mesh size of $2\times2\times2\;\mathrm{nm}^3$.

Specifically, this simulation focuses on a sample with dimensions $80\times80\times6\;\mathrm{nm}^3$. One observation of notable interest is the onset of a dynamic instability in the LL equation, attributed to the presence of the DMI. Starting from an initial uniform state of $\mathbf{m}^{0} = (0,0,1)^T$, the system relaxes into different stable configurations in response to variations in the damping parameter.
\begin{figure}[htbp]
  \centering
  \subfloat[$Q = -1$]{\includegraphics[width=1.6in]{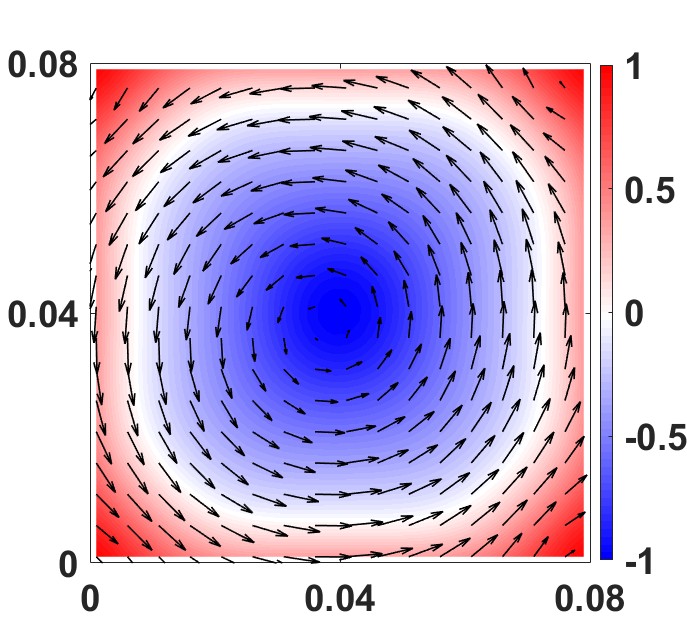}}
  \subfloat[$Q = 0$]{\includegraphics[width=1.6in]{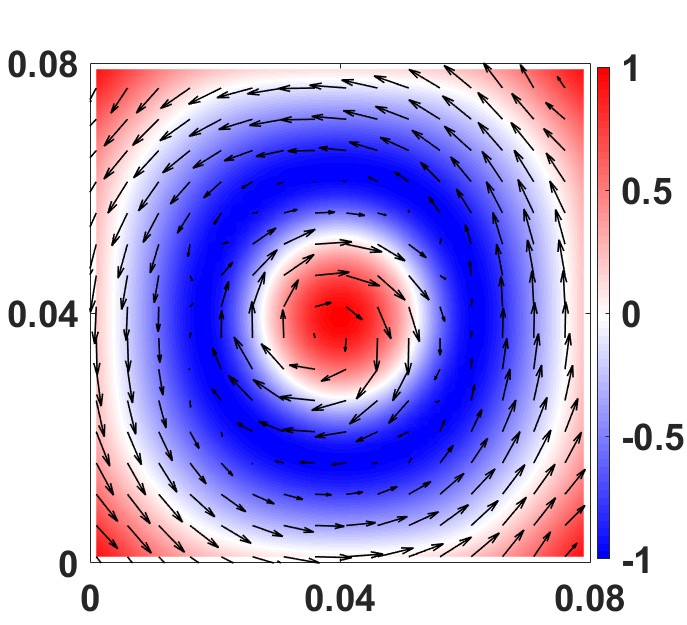}}
  \subfloat[$Q = 0$]{\includegraphics[width=1.6in]{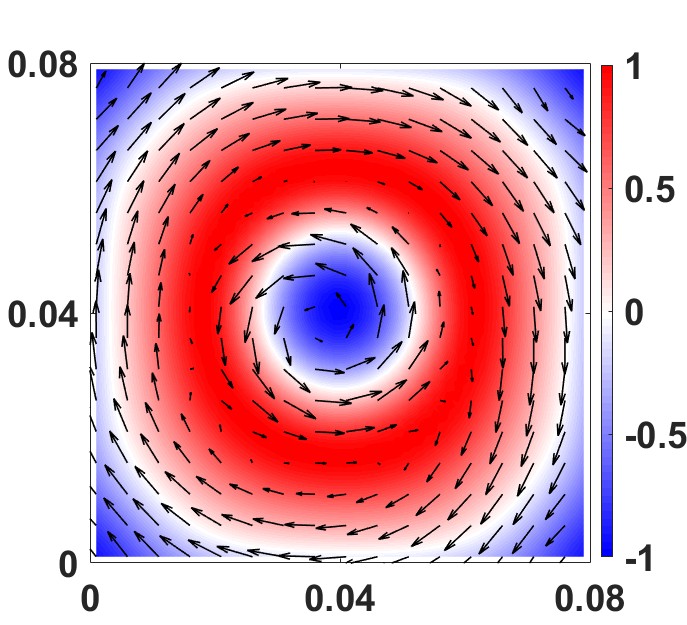}}
  \subfloat[$Q = 1$]{\includegraphics[width=1.6in]{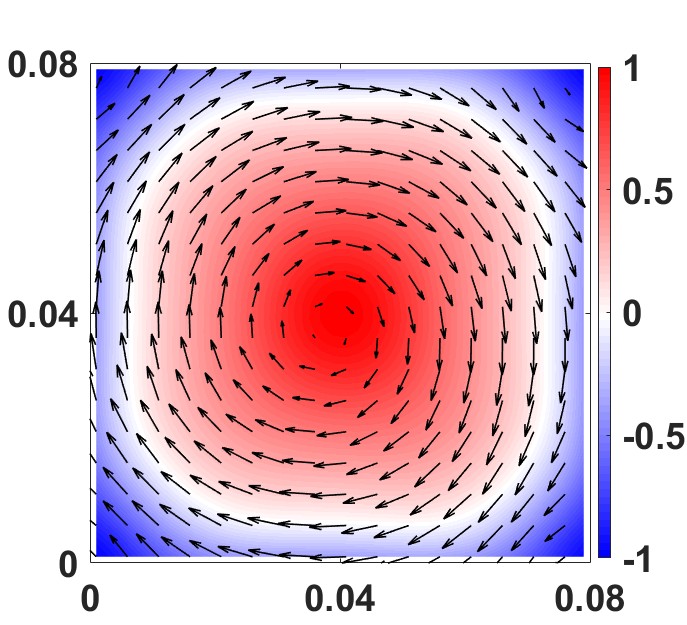}}
  \caption{Isolated skyrmions and isolated skyrmioniums by means of the LL equation with different damping parameters $\alpha = 0.05, 0.07, 0.2, 0.6$. The color of background represents the component $m_3$ and arrows represent the in-plane components $m_1$ and $m_2$. }
  \label{fig:stable-skyrmion}
\end{figure}
As depicted in \cref{fig:stable-skyrmion}, the LL equation readily generates skyrmions ($Q = \pm1$) and skyrmioniums ($Q = 0$) with diverse damping parameters $\alpha$. We also document the energies and spatially averaged magnetization of the four configurations accounting for skyrmions and skyrmioniums in \cref{tab:isolated-sky-skym}. For varied $\alpha$ values within the interval $(0, 1]$, the system reliably converges towards one of the four configurations.
\begin{table}[h]
  \centering
  \caption{Energy and spatially averaged magnetization $\langle\mathbf{m}\rangle = (\langle m_1\rangle, \langle m_2\rangle, \langle m_3\rangle)^T$ of isolated skyrmions and isolated skyrmioniums.}
  \begin{tabular}{||c|c|c|c|c|c||}
    \hline
    Label & energy($10^{-18}\;\mathrm{J}$) & $\langle m_x\rangle$ & $\langle m_y\rangle$ & $\langle m_z\rangle$ & $\alpha$ \\
    \hline
    \cref{fig:stable-skyrmion}(A) & -3.8989 & 0.0047 & 0.0038 & -0.1204 & 0.05/0.06 \\
    \hline
    \cref{fig:stable-skyrmion}(B) & -3.0936 & -0.0035 & 0.0036 & -0.2627 & 0.07/0.08/0.09/0.1 \\
    \hline
    \cref{fig:stable-skyrmion}(C) & -3.0938 & 0.0188 & -0.0038 & 0.2629 & 0.2/0.3/0.4/0.5 \\
    \hline
    \cref{fig:stable-skyrmion}(D) & -3.8986 & 0.0021 & 0.0042 & 0.1216 & 0.01/$\cdots$/0.04/0.6/$\cdots$/1.0 \\
    \hline
  \end{tabular}
  \label{tab:isolated-sky-skym}
\end{table}

The publication~\cite{PhysRevB.94.094420} reveals that a skyrmionium is a composite structure composed of two topological magnetic skyrmions possessing $Q = 1$ and $Q = -1$, and its motion is swifter when driven by an external out-of-plane current than that of a skyrmion. In order to gain a better understanding of their fundamental differences, the energy density distribution is visualized. Let $\mathcal{L}(\mathbf{m})$ denote the energy density distribution in the absence of the Dzyaloshinskii-Moriya interaction (DMI), and $\mathcal{D}(\mathbf{m})$ denote the energy density distribution corresponding to the DMI. The combination of these two, $\mathcal{T}(\mathbf{m}) = \mathcal{L}(\mathbf{m})+\mathcal{D}(\mathbf{m})$, represents the total energy density distribution. As illustrated in \cref{fig:eng_dis_iso_sky_skym}, the difference between the skyrmion and skyrmionium is mainly attributable to $\mathcal{D}(\mathbf{m})$, as their energy distributions without the DMI appear similar along the axes passing through the center. However, the energy density distribution of the DMI is almost entirely opposite.
\begin{figure}[ht]
  \centering
  \subfloat[$\mathcal{L}(\mathbf{m})$]{\includegraphics[width=1.8in]{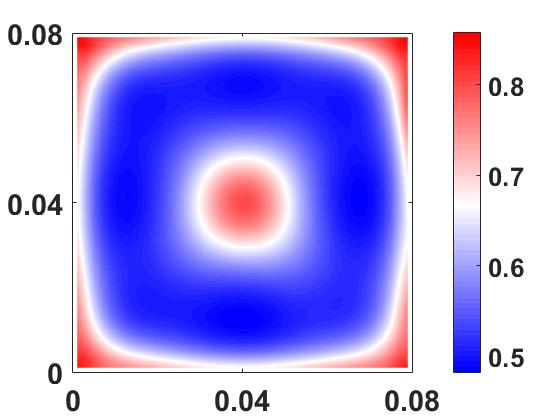}}
  \subfloat[$\mathcal{D}(\mathbf{m})$]{\includegraphics[width=1.8in]{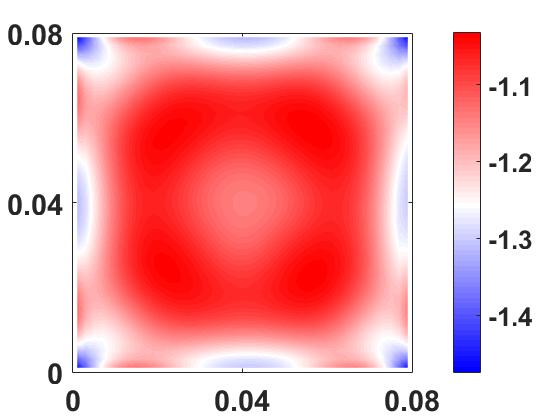}}
  \subfloat[$\mathcal{T}(\mathbf{m})$]{\includegraphics[width=1.8in]{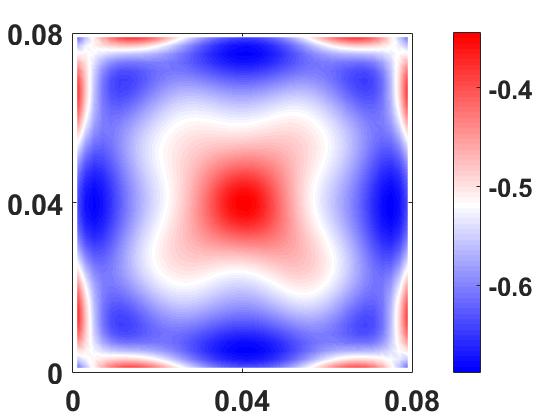}}
  \quad
  \subfloat[$\mathcal{L}(\mathbf{m})$]{\includegraphics[width=1.8in]{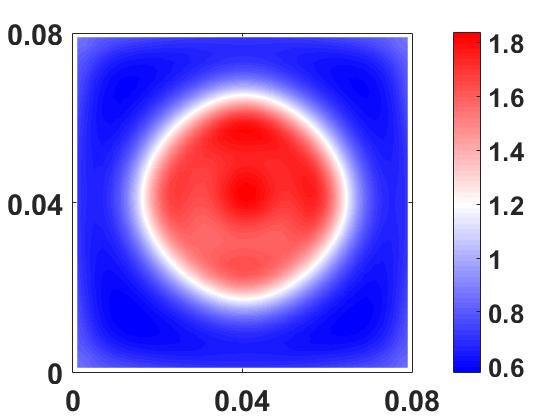}}
  \subfloat[$\mathcal{D}(\mathbf{m})$]{\includegraphics[width=1.8in]{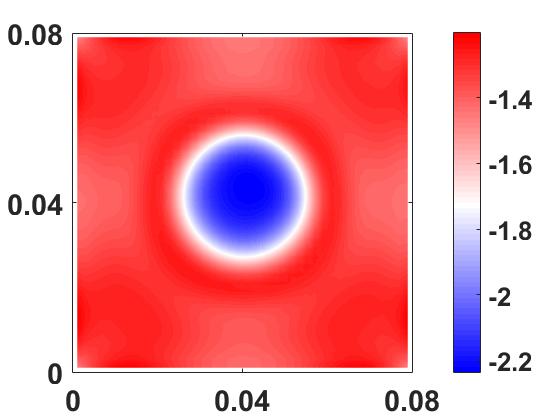}}
  \subfloat[$\mathcal{T}(\mathbf{m})$]{\includegraphics[width=1.8in]{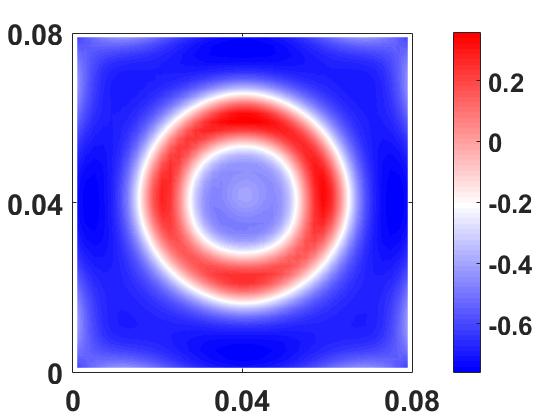}}
  \caption{The energy density distribution along the centered slice of the material in the $xy$-plane. Top row: energy density distribution of the skyrmion with $Q = 1$. Bottom row: energy density distribution of the skyrmionium.}
\label{fig:eng_dis_iso_sky_skym}
\end{figure}

In the context of the harmonic map heat flow equation, the relaxation of the system is aimed toward achieving a single skyrmion configuration. As portrayed in \cref{fig:80nm_compare_LLandGD}, a meticulous scrutiny of the quantity $\langle m_3\rangle$ indicates a swift formation of the skyrmion, followed by a prolonged period of relaxation that is required to satisfy the stability criterion.
\begin{figure}[h]
  \centering
\subfloat[Snapshots of relaxation driven by the harmonic map heat flow.]{
  \begin{overpic}
    [width=1.2in]{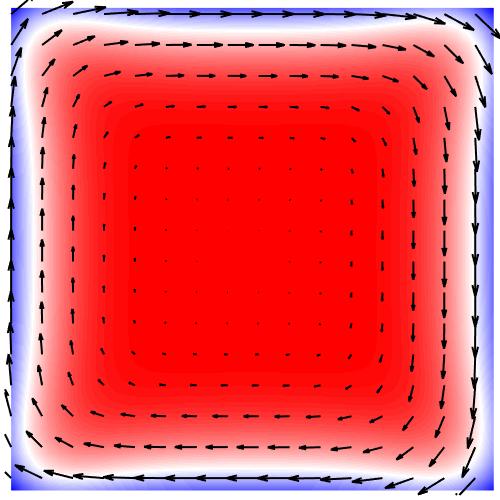}
    \put(5,80){\fcolorbox{black}{white}{\tiny{$10\;\mathrm{ps}$}}}
  \end{overpic}
  \begin{overpic}
    [width=1.2in]{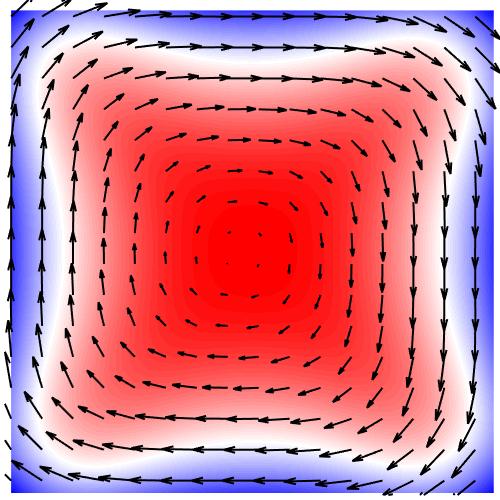}
    \put(5,80){\fcolorbox{black}{white}{\tiny{$20\;\mathrm{ps}$}}}
  \end{overpic}
  \begin{overpic}
    [width=1.2in]{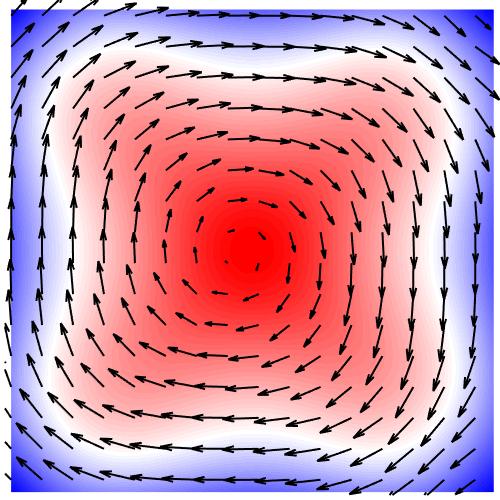}
    \put(5,80){\fcolorbox{black}{white}{\tiny{$45\;\mathrm{ps}$}}}
  \end{overpic}
  \begin{overpic}
    [width=1.2in]{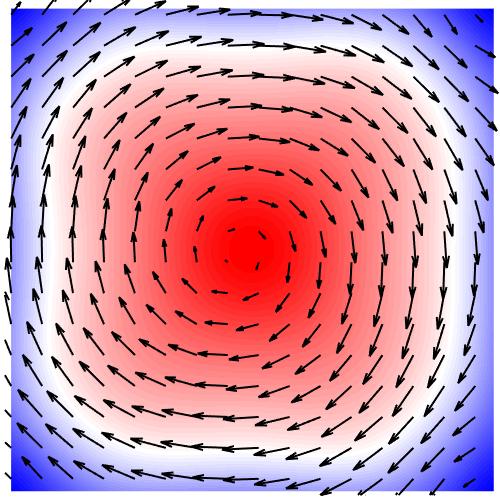}
    \put(5,80){\fcolorbox{black}{white}{\tiny{$100\;\mathrm{ps}$}}}
  \end{overpic}
  \begin{overpic}
    [width=1.2in]{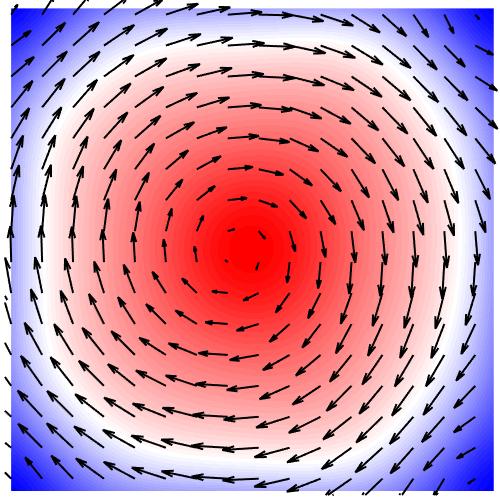}
    \put(5,80){\fcolorbox{black}{white}{\tiny{$350\;\mathrm{ps}$}}}
  \end{overpic}}
  \quad
  \subfloat[Snapshots of relaxation driven by the LL equation.]{\begin{overpic}
    [width=1.2in]{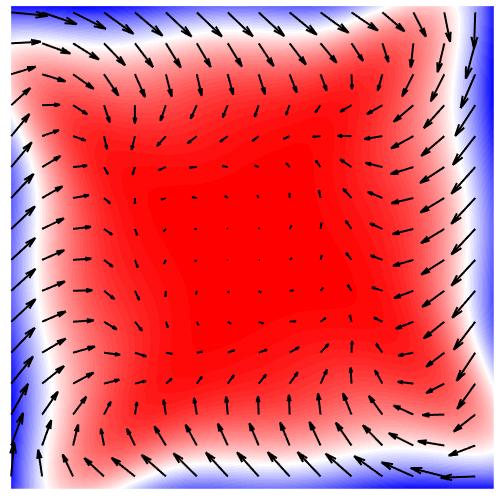}
    \put(5,80){\fcolorbox{black}{white}{\tiny{$10\;\mathrm{ps}$}}}
  \end{overpic}
  \begin{overpic}
    [width=1.2in]{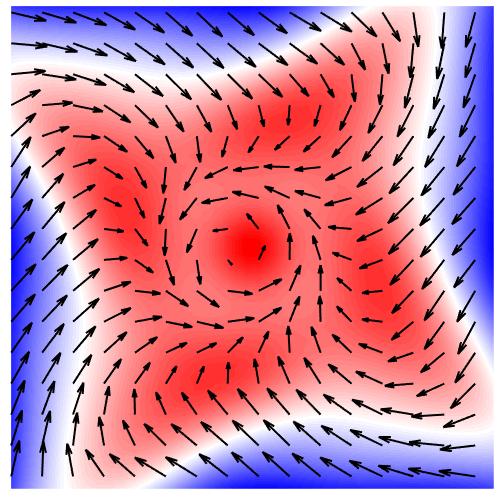}
    \put(5,80){\fcolorbox{black}{white}{\tiny{$20\;\mathrm{ps}$}}}
  \end{overpic}
  \begin{overpic}
    [width=1.2in]{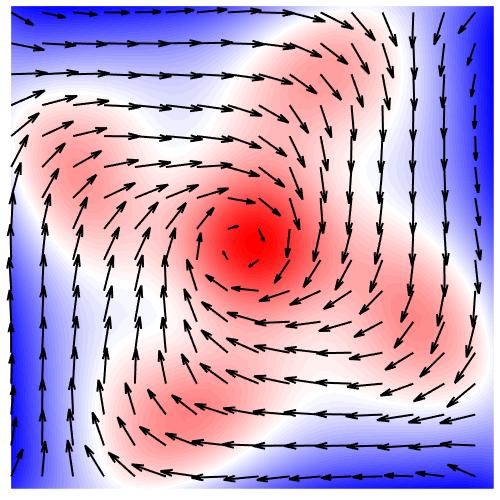}
    \put(5,80){\fcolorbox{black}{white}{\tiny{$45\;\mathrm{ps}$}}}
  \end{overpic}
  \begin{overpic}
    [width=1.2in]{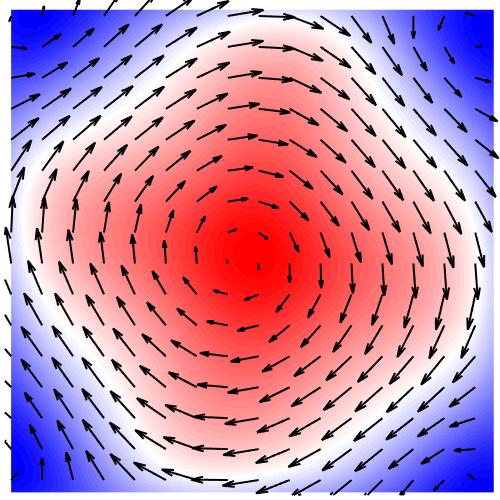}
    \put(5,80){\fcolorbox{black}{white}{\tiny{$100\;\mathrm{ps}$}}}
  \end{overpic}
  \begin{overpic}
    [width=1.2in]{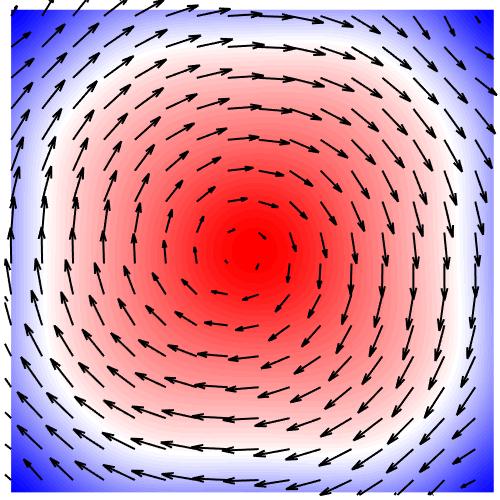}
    \put(5,80){\fcolorbox{black}{white}{\tiny{$350\;\mathrm{ps}$}}}
  \end{overpic}}
\quad
  \subfloat[Energy evolution.]{\includegraphics[width=3.0in]{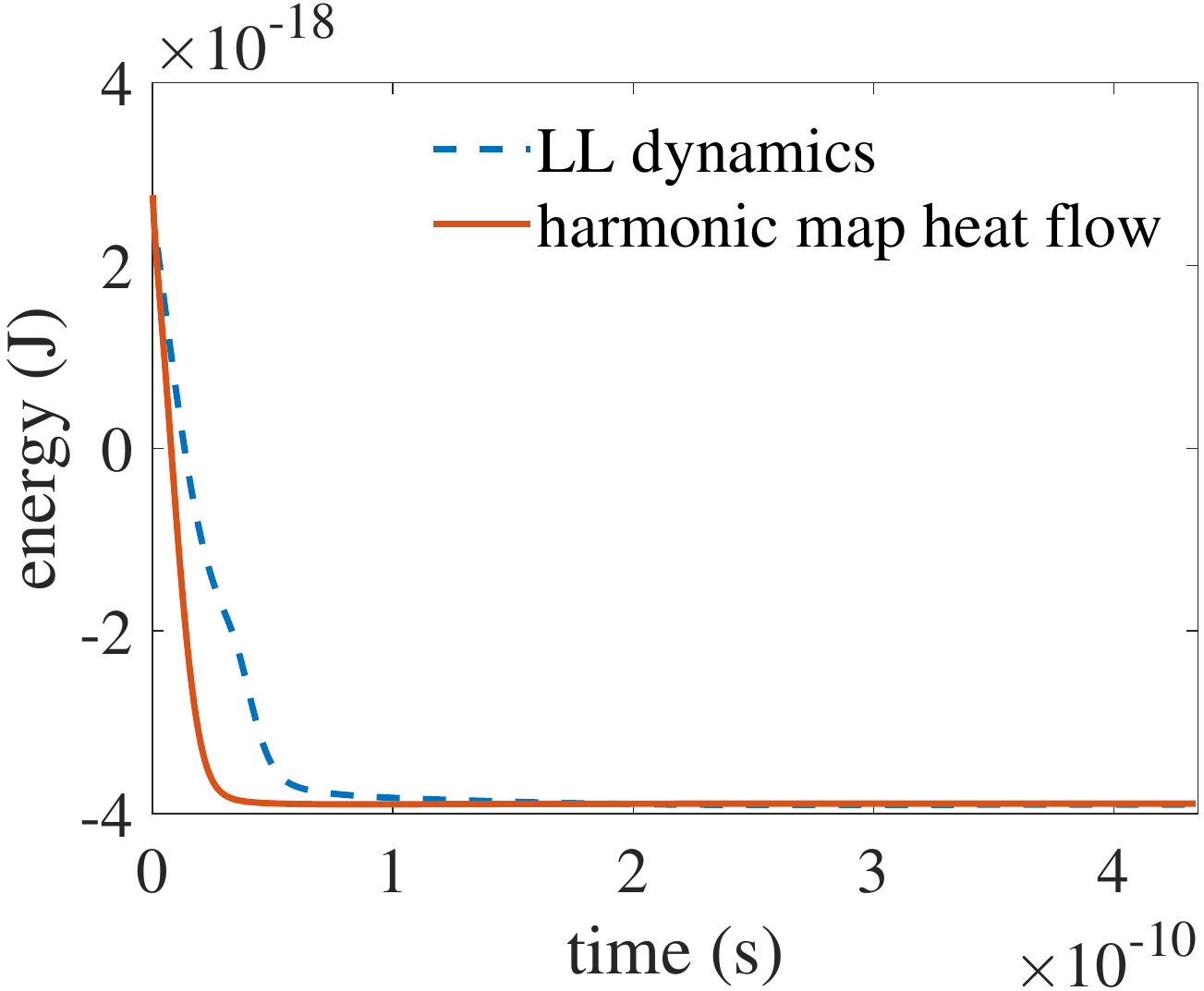}}
  \subfloat[$\langle\mathbf{m}\rangle$ evolution.]{\includegraphics[width=3.0in]{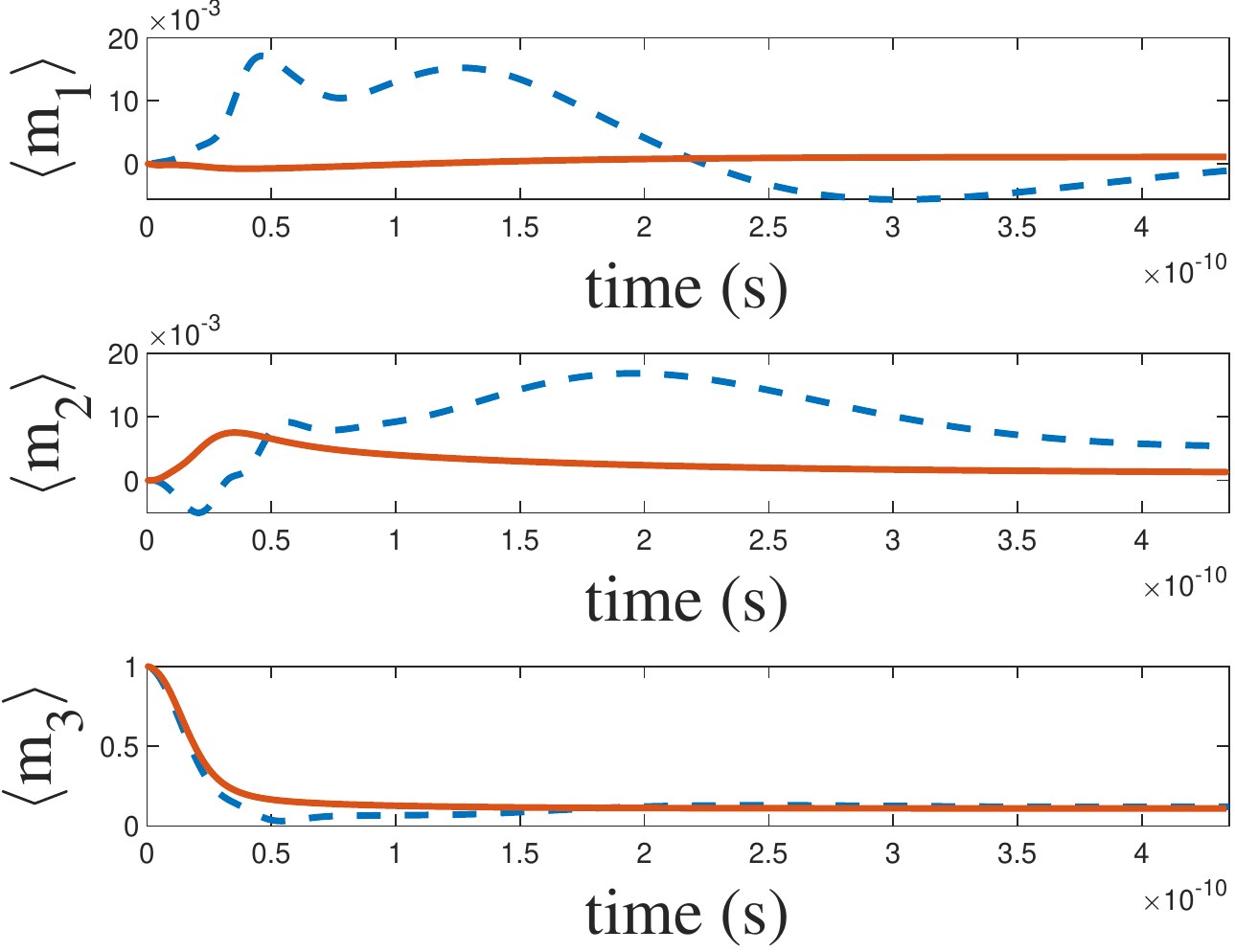}}
  \caption{Comparison of the relaxation driven by the LL equation and the harmonic map heat flow. The same initialization $\mathbf{m}^0 = (0,0,1)^T$ is used and the same skyrmion with $Q = 1$ is reached. The damping parameter $\alpha = 0.6$ is used in the LL equation.}
\label{fig:80nm_compare_LLandGD}
\end{figure}

Isolated skyrmion and skyrmionium structures have been successfully generated in a controllable manner. Subsequently, we have expanded our efforts towards generating skyrmion clusters in a ferromagnetic sample with dimensions of $200\times200\times6\;\mathrm{nm}^3$. Initiated from a configuration where $\mathbf{m}^0 = (0,0,1)^T$, the system undergoes relaxation to yield varying clusters as we adjust the damping parameter $\alpha$. A skyrmion cluster is characterized based on the number of skyrmions it comprises and the nature of their interconnected structures. \Cref{fig:stable-lattice-skyrmion} depicts representative skyrmion clusters generated by employing the LL equation.
\begin{figure}[h]
  \centering
  %\subfloat[$\alpha = 0.01$]{\includegraphics[width=1.3in]{200_200_6_dmpng_001.jpg}}
  \subfloat[$\alpha = 0.04$.]{\includegraphics[width=1.5in]{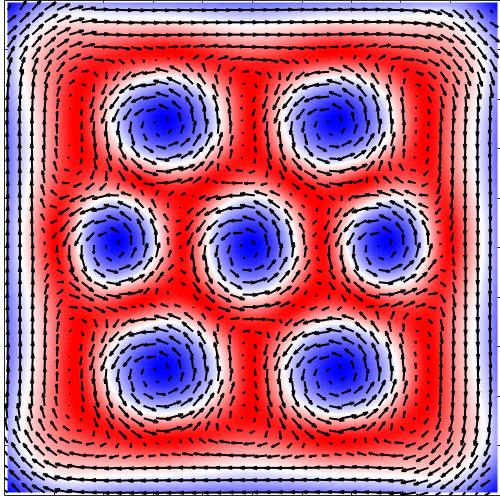}}
  \subfloat[$\alpha = 0.05$.]{\includegraphics[width=1.5in]{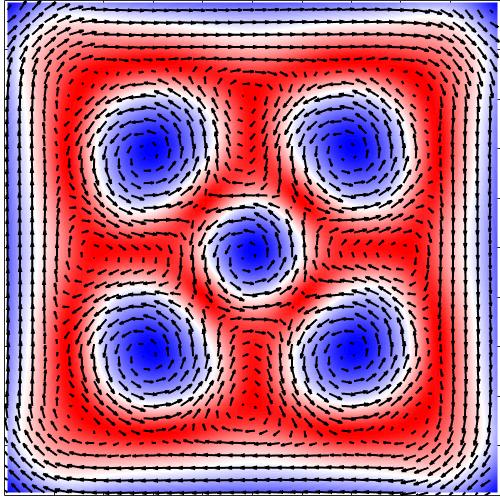}}
  %\subfloat[$\alpha = 0.06$]{\includegraphics[width=1.0in]{200_200_6_dmpng_006.jpg}}
  %\quad
  \subfloat[$\alpha = 0.09$.]{\includegraphics[width=1.5in]{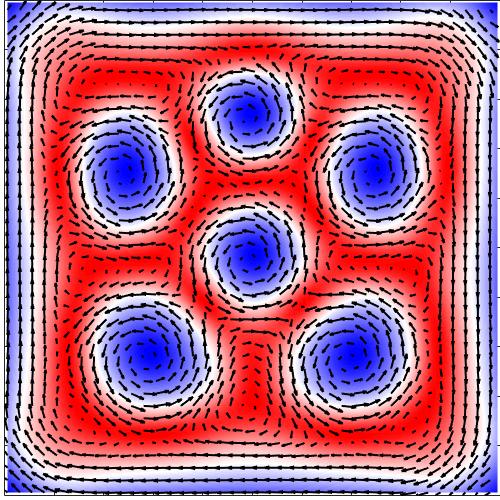}}
  %\subfloat[$\alpha = 0.2$]{\includegraphics[width=1.3in]{200_200_6_dmpng_02.jpg}}
%  \quad
  \subfloat[$\alpha = 0.5$.]{\includegraphics[width=1.5in]{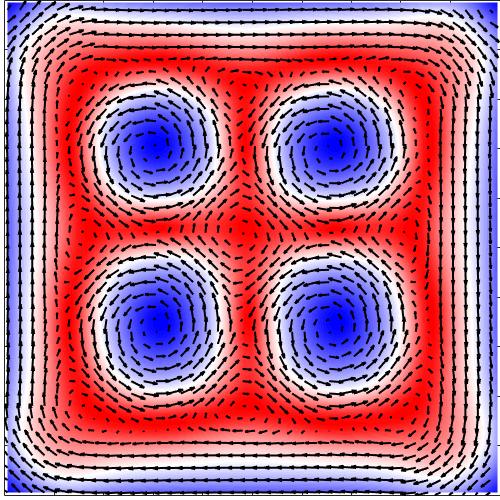}}
  %\subfloat[$\alpha = 0.6$]{\includegraphics[width=1.3in]{200_200_6_dmpng_06.jpg}}
%  \quad
  \caption{Representative skyrmion clusters formed in the $200\times200\times6\;\mathrm{nm}^3$ ferromagnet.}
  \label{fig:stable-lattice-skyrmion}
\end{figure}

The mutual interactions between individual skyrmions can lead to the formation of skyrmion lattices and clusters. As illustrated in \cref{fig:stable-lattice-skyrmion}, the local structure of skyrmion lattices can be realized by the presence of skyrmion clusters. For instance, by adopting a specific initialization scheme, the square skyrmion lattice structure can be easily generated as a periodic replica of the square skyrmion cluster, as shown in either \cref{fig:stable-lattice-skyrmion}(D) or \cref{fig:200nm-skyL}.
\begin{figure}[h]
  \centering
  \subfloat{\includegraphics[width=1.6in]{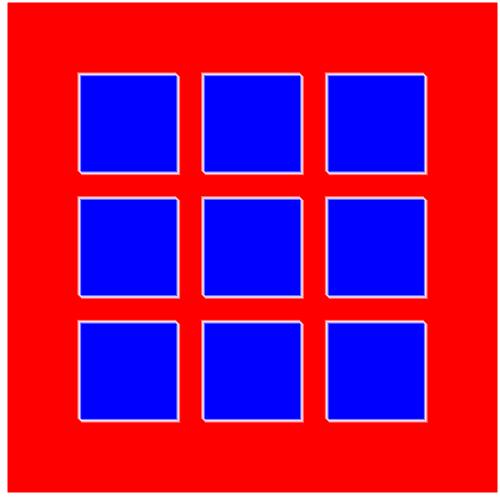}}\quad
\subfloat{\includegraphics[width=1.6in]{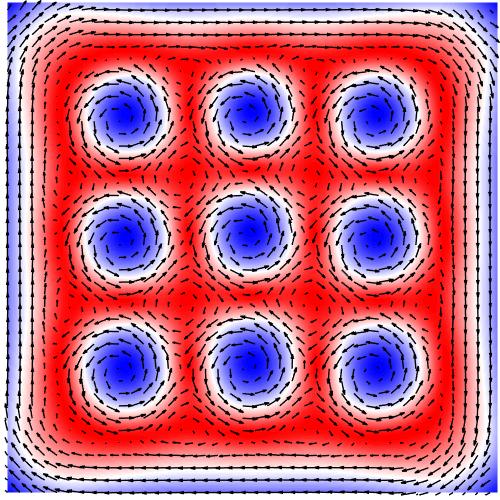}}\quad
\subfloat{\includegraphics[width=1.6in]{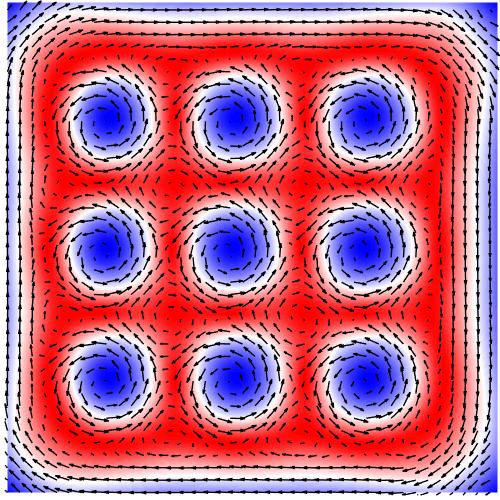}}
%\quad
%\subfloat[The relaxation of LL dynamics]{\includegraphics[width=2.0in]{200nm_LLG_regular_skyrmions_lattice_mag.pdf}
%\includegraphics[width=2.0in]{200nm_force_regular_skyrmions_lattice_mag.pdf}}
\caption{Left: Initial magnetization configuration. Middle and right: Stable magnetization configurations formed by the LL equation and harmonic map heat flow, respectively. The back ground color represents the magnetization component $m_3$. The magnetization within the blue blocks is $(0,0,-1)^T$, while the remianing area is $(0,0,1)^T$.}
\label{fig:200nm-skyL}
\end{figure}
The initial magnetization exhibits a rectangular shape with in-plane dimensions of $40\;\mathrm{nm}\times40\;\mathrm{nm}$, a configuration which yields skyrmions having a diameter of $40\;\mathrm{nm}$. For the skyrmion lattice generated via the LL equation, the stable energy and spatially averaged magnetization are $-1.7066\textrm{e-}17\;\mathrm{J}$ and $\langle\mathbf{m}\rangle = (-0.58\textrm{e-}04,0.39\textrm{e-}03,0.23)^T$, respectively. Meanwhile, the skyrmion lattice produced by means of the harmonic map heat flow method possesses stable energy and averaged magnetization values of $-1.6971\textrm{e-}17\;\mathrm{J}$ and $\langle\mathbf{m}\rangle = (0.28\textrm{e-}03,0.04,0.23)^T$, respectively.

The emergence of isolated skyrmioniums and skyrmionium clusters is a subject that has received little attention in the existing literature. In this study, we leverage the dynamic instability exhibited by the LL equation to create isolated skyrmioniums and skyrmionium clusters with precise damping parameter settings. When the ferromagnetic material attains an $L$ value of $80\;\mathrm{nm}$, both isolated skyrmions and skyrmioniums can be observed, with the latter's radius being considerably larger than that of the former, specifically in an unsaturated phase. Consequently, we broaden our investigation by considering a sample of dimensions $500\times500\times6\;\mathrm{nm}^3$ and adopt two different initialization strategies to generate skyrmionium clusters.

Our first strategy involves placing nine rectangles with $m_3 = -1$ within the initial magnetization configuration and adjusting the inter-block distance. The dimensions of the blocks are fixed at $100\;\mathrm{nm}\times100\;\mathrm{nm}$, while the damping parameter is set to $\alpha=0.2$. When the blocks are spaced at $25\;\mathrm{nm}$, four skyrmions emerge near the corners along with a set of nine skyrmioniums arranged in the shape of a flower, as displayed in \cref{fig:500nm_mixed_skyrmions}(B). As the inter-block distance increases to $66\;\mathrm{nm}$, four skyrmions are generated at diverse locations in a regular skyrmionium lattice. It is important to note that the damping parameter plays a crucial role in the formation of skyrmionium clusters, unlike skyrmion clusters and skyrmion lattices. Furthermore, in contrast to skyrmion clusters and lattices, individual skyrmions consistently coexist with skyrmioniums in the skyrmionium clusters.
\begin{figure}[ht]
  \centering
  \subfloat[ ]{\includegraphics[width=1.8in]{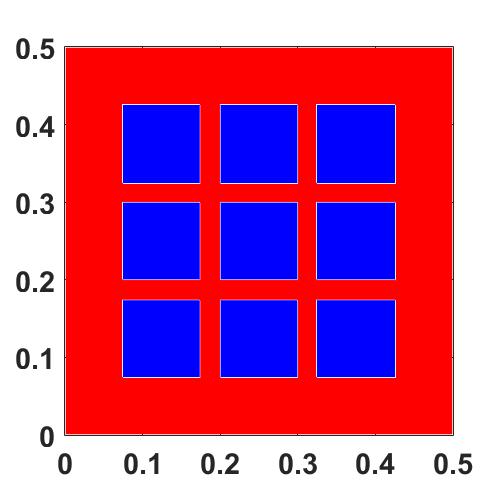}}
  \subfloat[ ]{\includegraphics[width=1.8in]{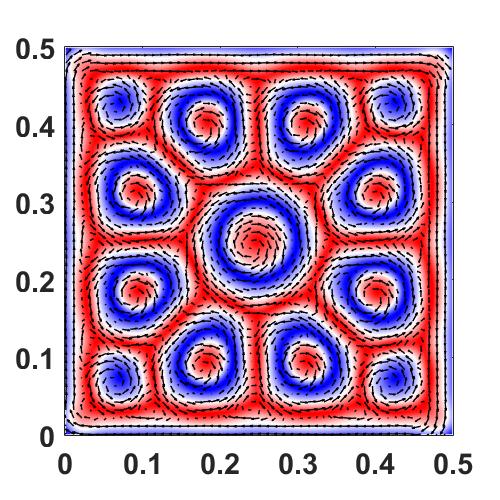}}
  \subfloat[ ]{\includegraphics[width=1.8in]{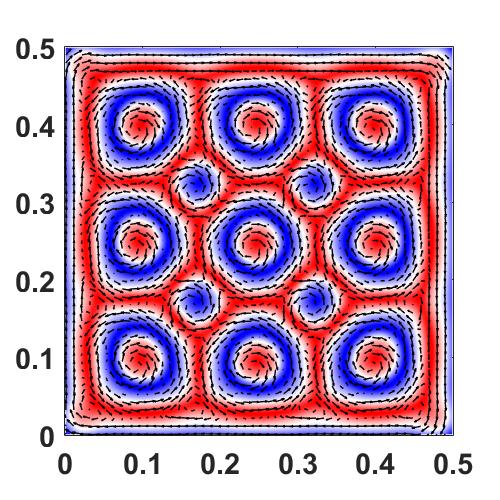}}
  \caption{Initial configuration (A) and skyrmionium clusters (B) and (C). The sample size is $500\;\mathrm{nm}\times 500\;\mathrm{nm}\times6\;\mathrm{nm}$ and the damping parameter is $\alpha = 0.2$ in the LL equation.}
  \label{fig:500nm_mixed_skyrmions}
\end{figure}

Next, we reduce the size of the blocks in the initial configuration to $50\;\mathrm{nm}\times50\;\mathrm{nm}$ and set $\alpha = 0.1$ in the LL equation. As illustrated in \cref{fig:500nm-lattice}, a skyrmion lattice comprising of 29 skyrmions and an additional isolated skyrmion located at a corner is produced. However, we observe a defective lattice in this case due to the existence of three distinct types of skyrmion clusters. Specifically, the skyrmion clusters are categorized into three types: (1) a skyrmion is encircled by five neighboring skyrmions; (2) a skyrmion is surrounded by six neighboring skyrmions; and (3) a skyrmion is surrounded by seven neighboring skyrmions. The occurrence of the first two types of clusters is also observed in a ferromagnetic material with dimensions of $200\;\mathrm{nm}\times 200\;\mathrm{nm}\times6\;\mathrm{nm}$.
\begin{figure}[ht]
  \centering
  \subfloat[Initialization]{\includegraphics[width=1.8in]{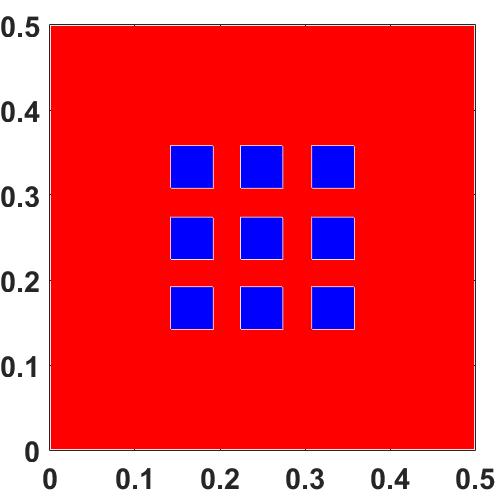}}
  \subfloat[Equilibrium]{\includegraphics[width=1.8in]{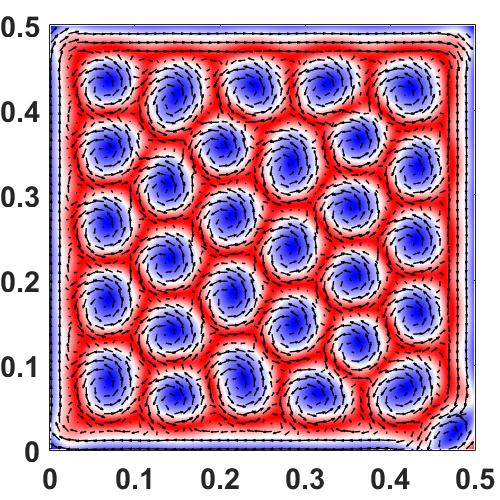}}
  \subfloat[$\mathcal{T}(\mathbf{m})$]{\includegraphics[width=1.88in]{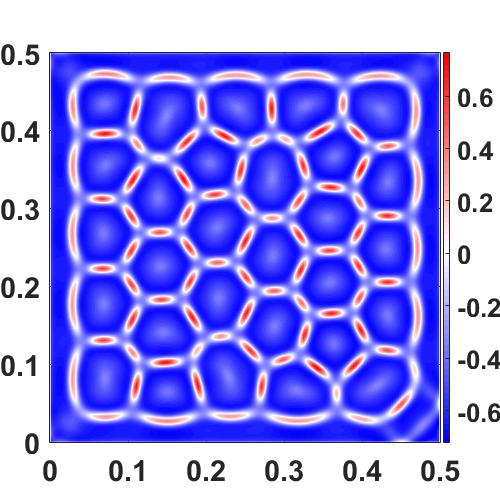}}
  \caption{A skyrmion lattice over a $500\;\mathrm{nm}\times500\;\mathrm{nm}\times6\;\mathrm{nm}$ ferromagnet. Here an isolated skyrmion is presented near the bottom-right corner, which resulted from the defectiveness of the sample. From the initial magnetization configuration (A), the system reaches (B) with spatial energy density distribution (C), following the LL dynamics with $\alpha = 0.1$.}
  \label{fig:500nm-lattice}
\end{figure}
When the isolated skyrmions coalesce into an interconnected structure, the energy density at the junctions of any two skyrmions exhibits a notably higher magnitude when compared to other locations.

Comprehending the phase transition between magnetic textures holds significant importance in the field of spintronics. Pertaining to skyrmion-based textures, experimental observations have revealed various transitions such as those between skyrmion clusters facilitated by a magnetic field~\cite{doi:10.1073/pnas.1600197113}, transitions between skyrmion lattice structures induced by a magnetic field~\cite{doi:10.1126/sciadv.1602562}, transitions between skyrmioniums driven by spin-polarized current~\cite{PhysRevB.94.094420}, and the formation of skyrmions via ultrafast laser pulses~\cite{PhysRevLett.110.177205}. The local order in magnetization is disrupted and then re-established to generate skyrmion and skyrmionium structures as demonstrated in~\cite{PhysRevLett.110.177205}. Motivated by these experimental findings, we aim to investigate the generation of skyrmion and skyrmionium structures in a skyrmion lattice configuration during the re-stabilization process through simulations.
\begin{figure}[h]
  \centering
  \subfloat[Magnetization within the centered circle with radius $100\;\mathrm{nm}$ is randomly revalued and re-stabilized.]{\includegraphics[width=1.8in]{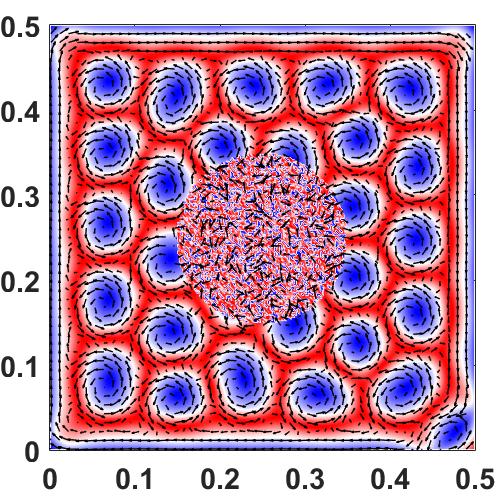}
\includegraphics[width=1.8in]{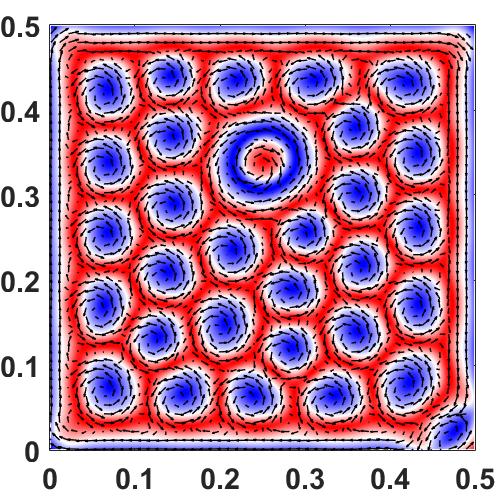}
  \includegraphics[width=1.9in]{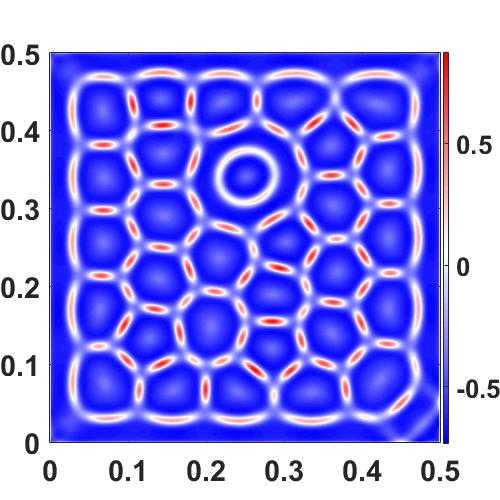}}
  \quad
  \subfloat[Magnetization within the centered circle with radius $120\;\mathrm{nm}$ is randomly revalued and re-stabilized.]{\includegraphics[width=1.8in]{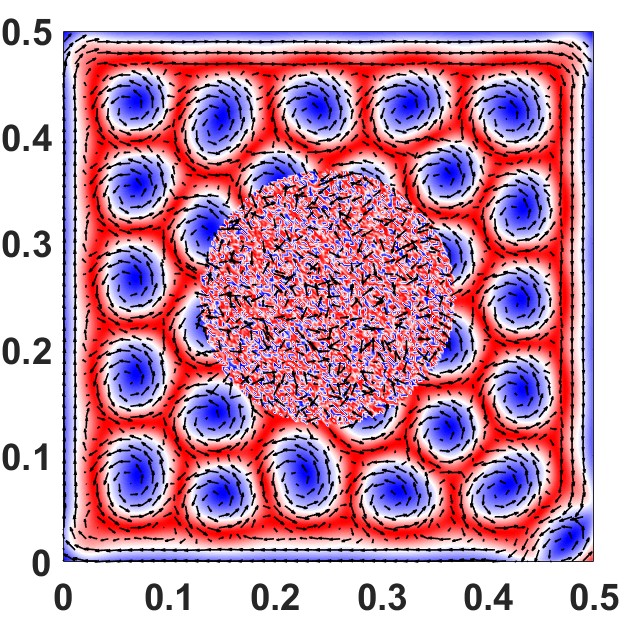}
\includegraphics[width=1.8in]{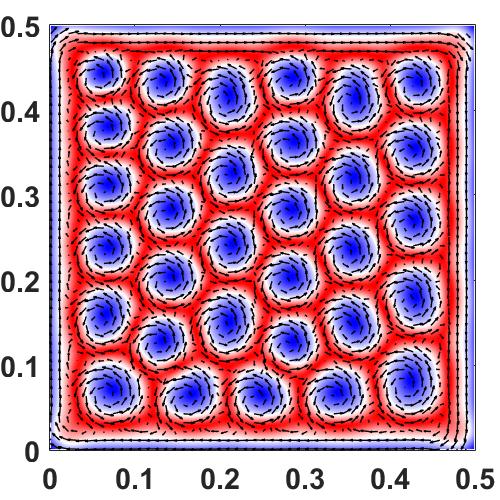}
  \includegraphics[width=1.9in]{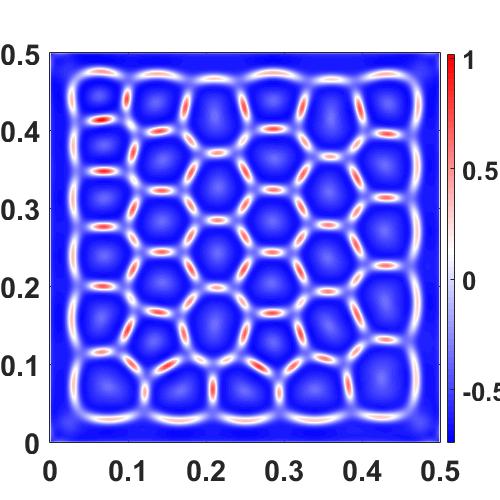}}
  \caption{Generation of skyrmioniums and skyrmions in a skyrmion lattice under pertubations. The initial magnetization configuration is a skyrmion lattice with the magnetization over a centered circle randomly pertubated, which is fed into the LL dynamics as the initial condition. A new stable magnetization configuration is obtained by following the LL equation. First column: initial configuration. Second column: stable magnetization configuration. Third column: energy density distribution of the stable magnetization configurations.}
\label{fig:random-initialization}
\end{figure}

In our simulations, we perturbe the magnetization order of the stable magnetic texture depicted in \cref{fig:500nm-lattice} by locally revaluating it randomly, in accordance with the LL equation. Specifically, in \cref{fig:random-initialization}, the magnetization order within a centered circular domain was randomly revalued while the radius of the circle was adjusted. The results demonstrate that this perturbation spontaneously induces the formation of either a skyrmionium or several skyrmions, thus initiating the transition of the skyrmion structure. A damping value of $\alpha = 0.1$ was employed, as depicted in \cref{fig:random-initialization} where the generation of the skyrmionium is observed as a consequence of the aforementioned perturbation.

\section{Transition paths}
\label{sec:application-string-method}

The collapse of an isolated skyrmion, its subsequent escape through a boundary, and division into two identical skyrmions, have been observed. The first two transitions shed light on the connection between isolated skyrmions and the classical magnetic saturation state, while the last transition doubles the system's topological number. This study aims to identify transition paths pertaining to changes in the topological number of both isolated skyrmions and skyrmion clusters. To this end, the string method, previously described in literature, has been employed to identify transition paths between magnetic textures with differing topological numbers. Subsequently, micromagnetic simulations have been conducted to realize the phase transition corresponding to the identified path.

The primary objective of this study is to investigate the transition between a skyrmion and a skyrmionium. By considering an isolated skyrmionium with a topological number $Q = 0$ and an isolated skyrmion with $Q=1$ as the two initial endpoints, a transition path can be established through the implementation of the string method. Such a path is illustrated in \cref{fig:isolsted_skym2sky}, where stable and saddle points are highlighted. Specifically, the transition from the skyrmionium to skyrmion with $Q=-1$ is accompanied by the escape of a nucleus located at the core of the skyrmionium from the protective ring pattern. It is worth noting that the nucleus of the skyrmionium can be regarded as a reversed skyrmion and can be easily controlled by an in-plane current, while the external ring cannot be as readily removed due to the pronounced boundary protection. Hence, this study provides valuable insights into the transition from an isolated skyrmionium to an isolated skyrmion, whereby the topological number is effectively erased as a result of the removal of the nucleus skyrmion.
\begin{figure}[h]
  \centering
\begin{overpic}
  [width=6.5in]{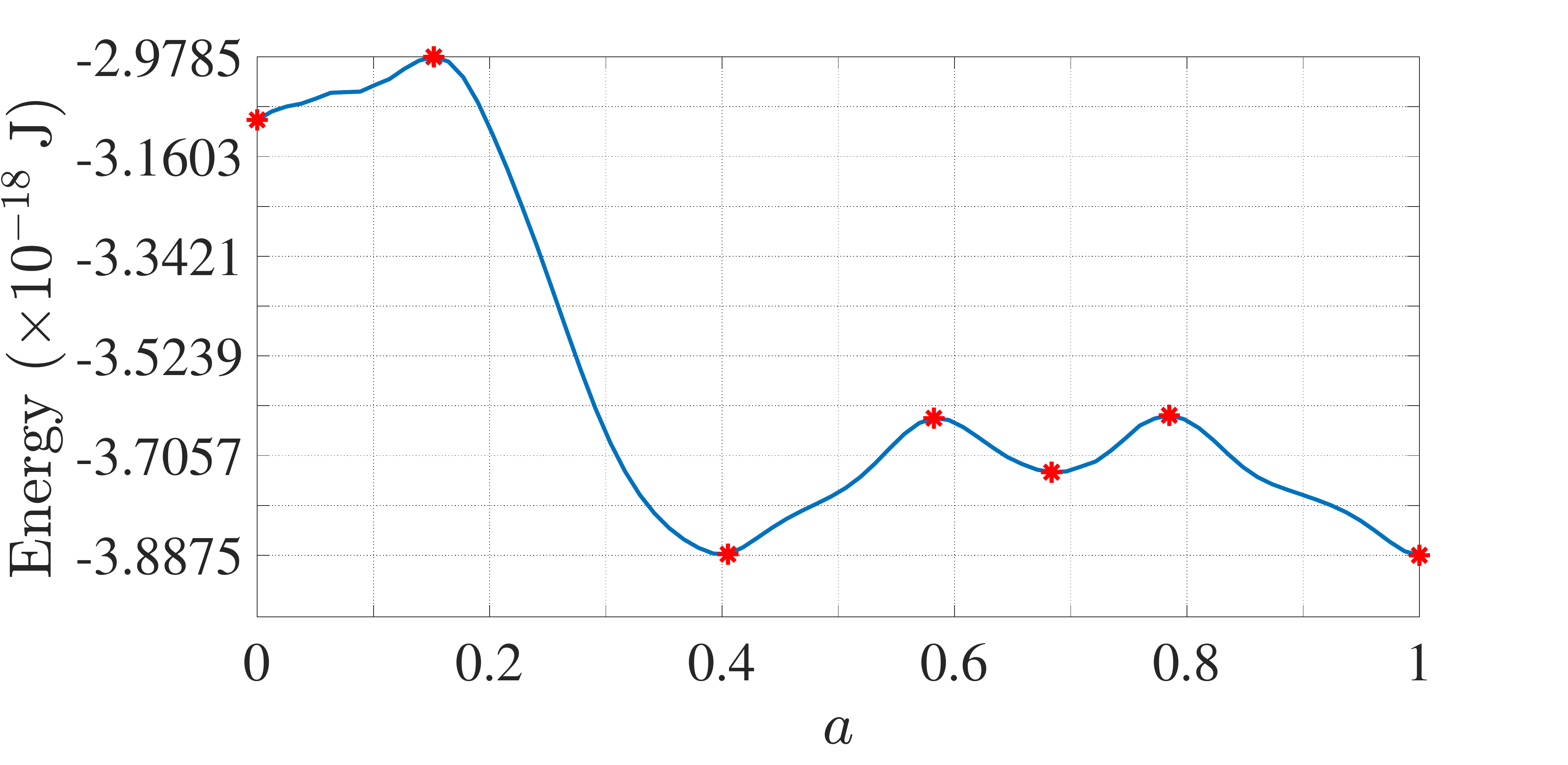}
\put(20,15){\fcolorbox{black}{white}{A}}
\put(18,39){B}
\put(26,42){C}
\put(45,15){D}
\put(58.5,24){E}
\put(66.5,20){F}
\put(73,24){G}
\put(87,11){H}
\end{overpic}
  \\
\begin{overpic}
  [width=0.8in]{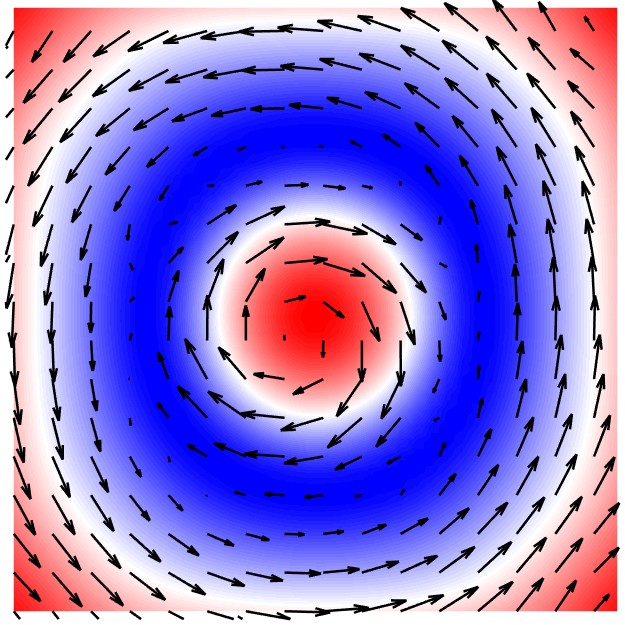}
\put(2,80){\fcolorbox{black}{white}{\tiny{B}}}
\end{overpic}
\begin{overpic}
  [width=0.8in]{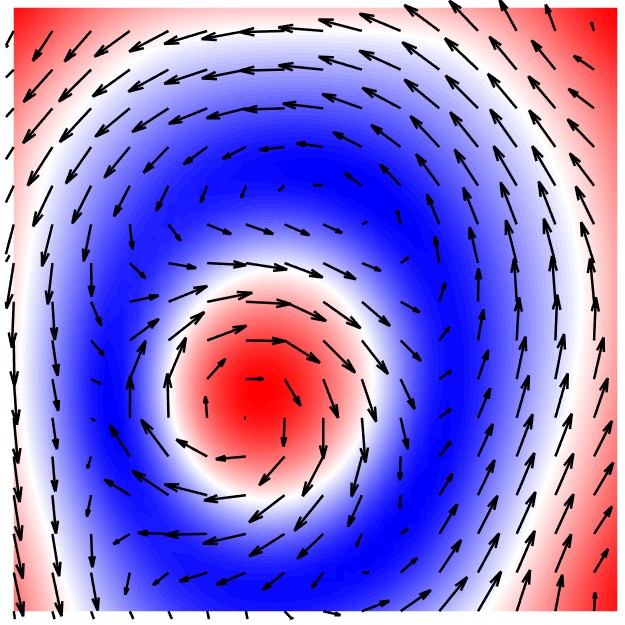}
\put(2,80){\fcolorbox{black}{white}{\tiny{C}}}
\end{overpic}
\begin{overpic}
  [width=0.8in]{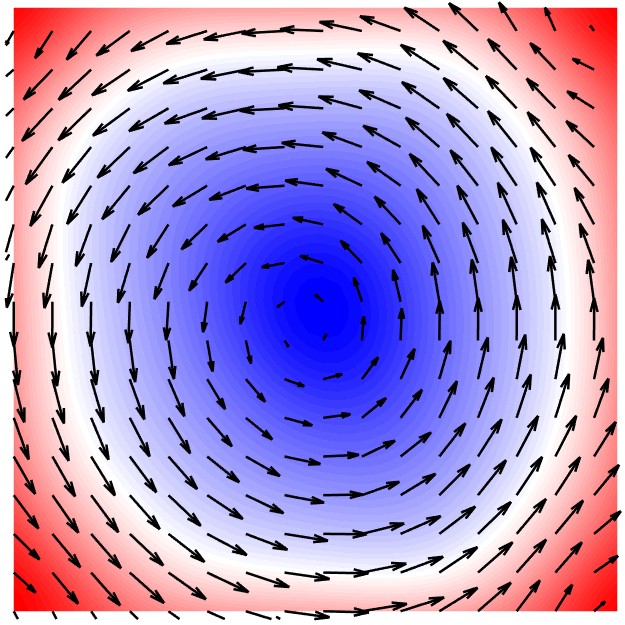}
\put(2,80){\fcolorbox{black}{white}{\tiny{D}}}
\end{overpic}
\begin{overpic}
  [width=0.8in]{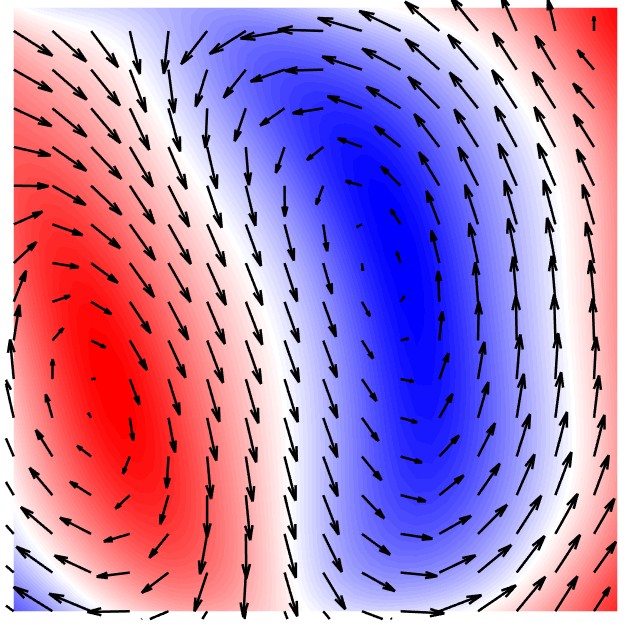}
\put(2,80){\fcolorbox{black}{white}{\tiny{E}}}
\end{overpic}
\begin{overpic}
  [width=0.8in]{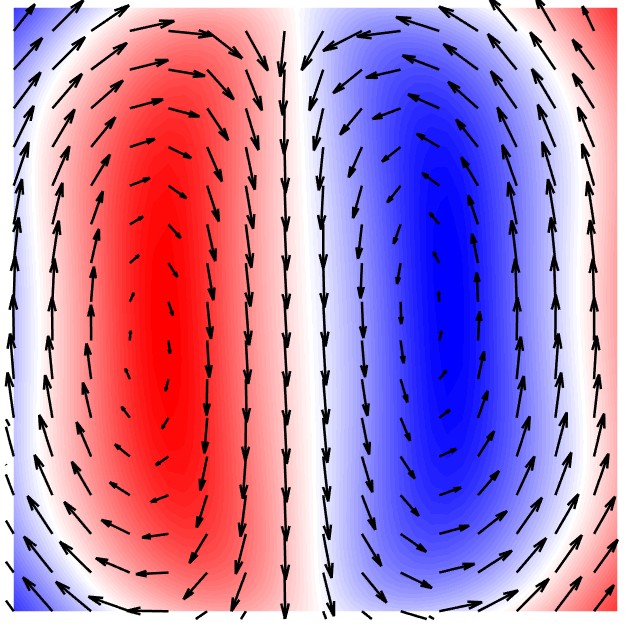}
\put(2,80){\fcolorbox{black}{white}{\tiny{F}}}
\end{overpic}
\begin{overpic}
  [width=0.8in]{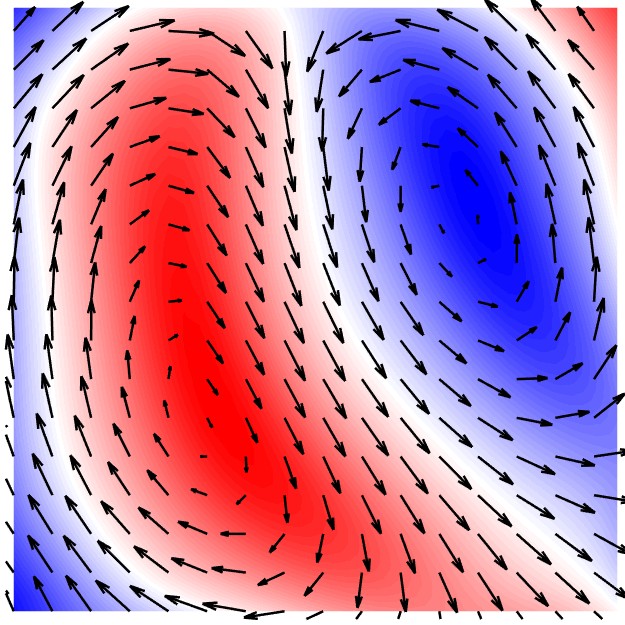}
\put(2,80){\fcolorbox{black}{white}{\tiny{G}}}
\end{overpic}
\begin{overpic}
  [width=0.8in]{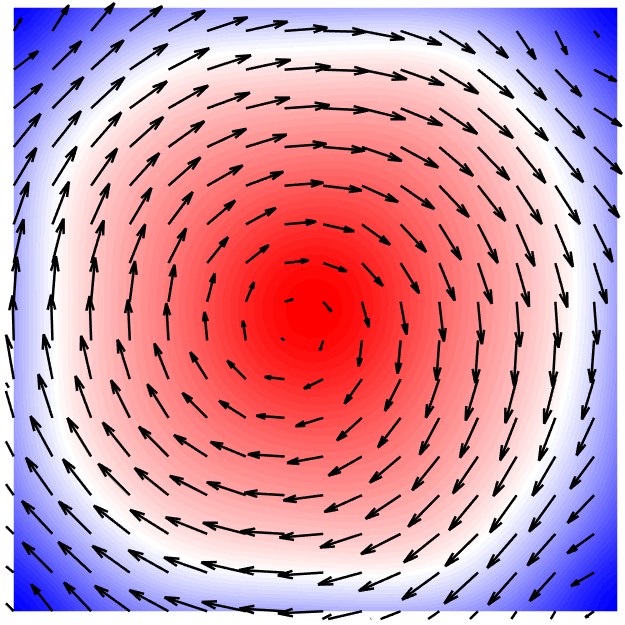}
\put(2,80){\fcolorbox{black}{white}{\tiny{H}}}
\end{overpic}
  \caption{Phase transition between an isolated skyrmionium $Q = 0$ and an isolated skyrmion $Q = 1$. (A) The transition path between the skyrmionium and the skyrmion. The red points denote the local minima and saddle points on the path. (B)-(H) are the magnetization configurations corresponding to minima and saddle points on the MEP.}
\label{fig:isolsted_skym2sky}
\end{figure}
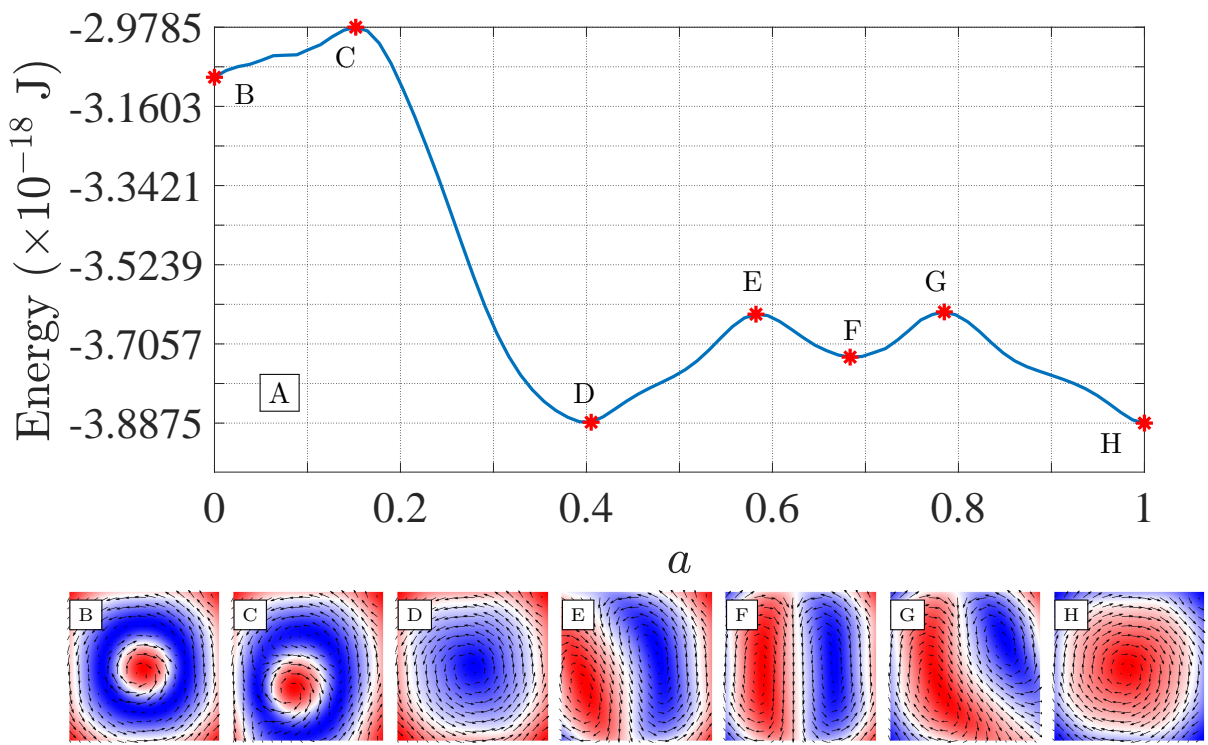

During the transition path between a skyrmion with a topological number of $Q = -1$ and one with $Q = 1$, a metastable state characterized by a skyrmion junction with a topological number of $Q = 0$ is encountered. The change in topological number follows the sequence $\pm1\rightarrow 0\rightarrow\mp 1$ along this transition path, as demonstrated in \cref{fig:80nm_simulations_STT}. In order to realize the transition from a skyrmion with $Q = 1$ to one with $Q = -1$ in the LL equation, an in-plane current is applied, with a chosen damping parameter of $\alpha = 0.6$. The simulation is conducted in two stages. Firstly, from $0\sim 850\;\mathrm{ps}$, a current with $u = -bJ = -150\;\mathrm{m/s}$ and $\beta = 0.5$ along the direction $-\mathbf{e}_1$ is applied, and the system reaches the skyrmion junction. Subsequently, from $1.5\;\mathrm{ns}$ to $2.1\;\mathrm{ns}$, a current with $u = -50\;\mathrm{m/s}$ and $\beta = 0.4$ is applied along the direction $-\mathbf{e}_2$, and the system relaxes to the skyrmion with $Q = -1$. The simulation results demonstrate that the energy barrier during the transition from a skyrmion to a skyrmion junction is higher than that from the skyrmion junction to the skyrmion, owing to the superior stability of isolated skyrmions.
\begin{figure}[h]
  \centering
  \includegraphics[width = 6.5in]{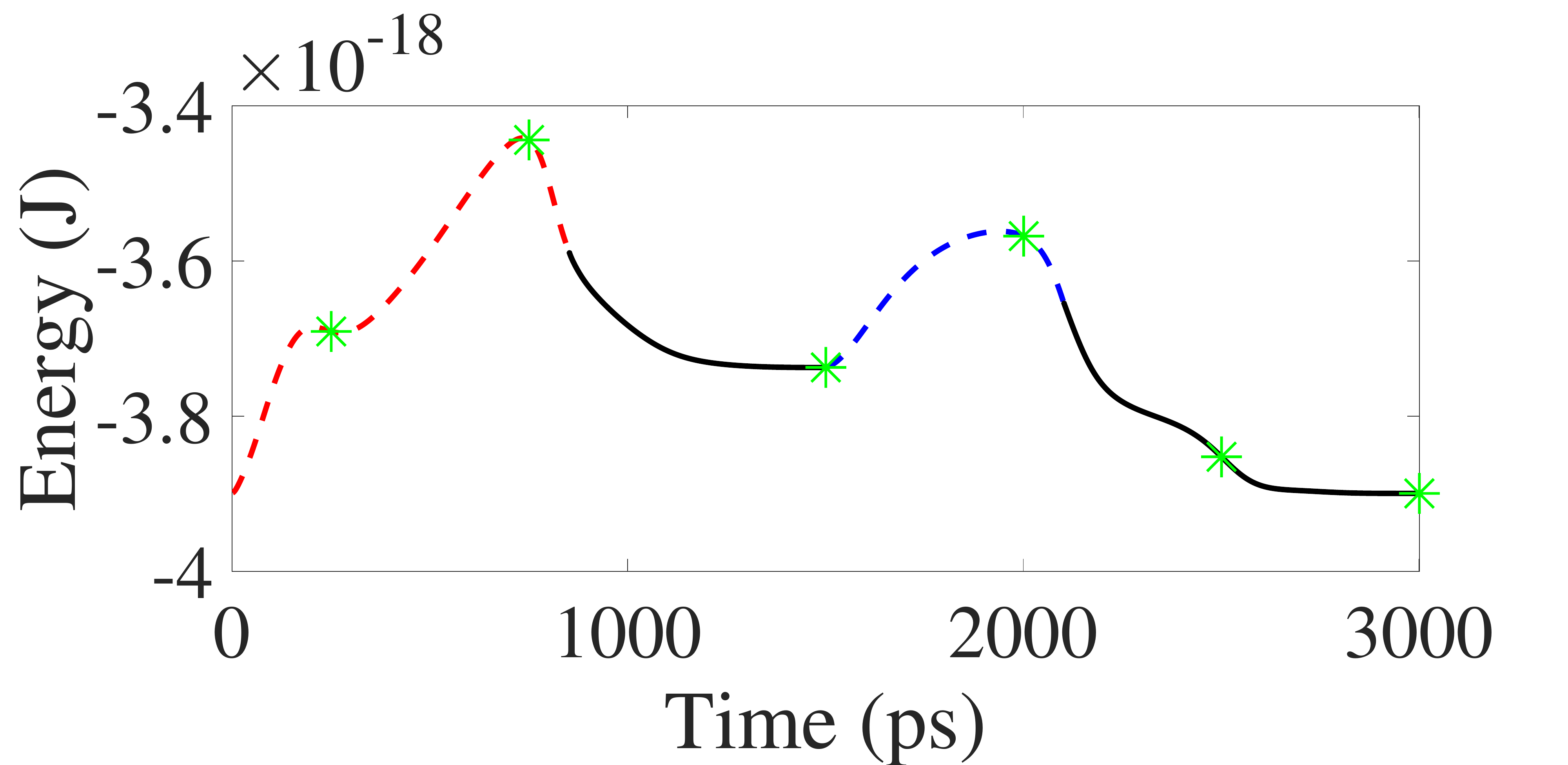}
  \\
  \begin{overpic}
    [width=.8in]{skym2sky_stataionary6.jpg}
    \put(5,80){\fcolorbox{black}{white}{\tiny{$0\;\mathrm{ps}$}}}
  \end{overpic}\begin{overpic}
    [width=.8in]{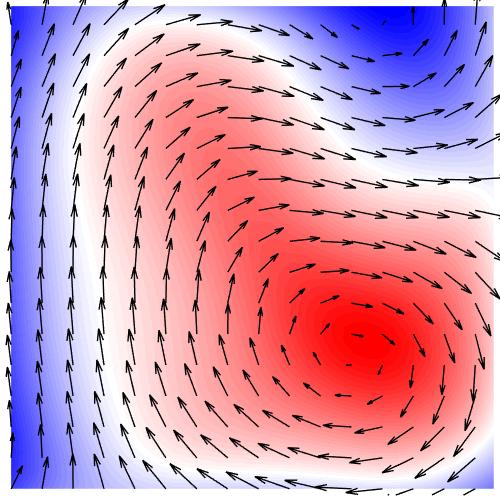}
    \put(5,80){\fcolorbox{black}{white}{\tiny{$250\;\mathrm{ps}$}}}
  \end{overpic}
  \begin{overpic}
    [width=.8in]{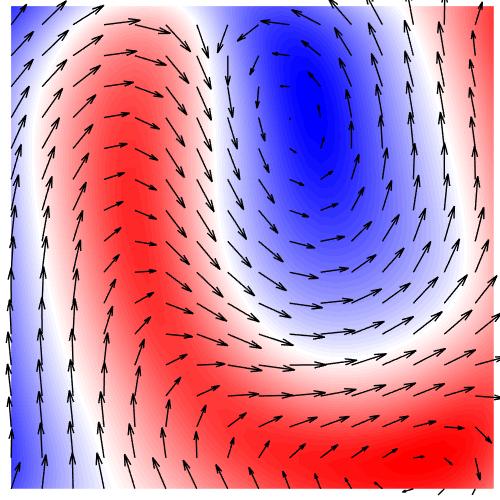}
    \put(5,80){\fcolorbox{black}{white}{\tiny{$750\;\mathrm{ps}$}}}
  \end{overpic}
  \begin{overpic}
    [width=.8in]{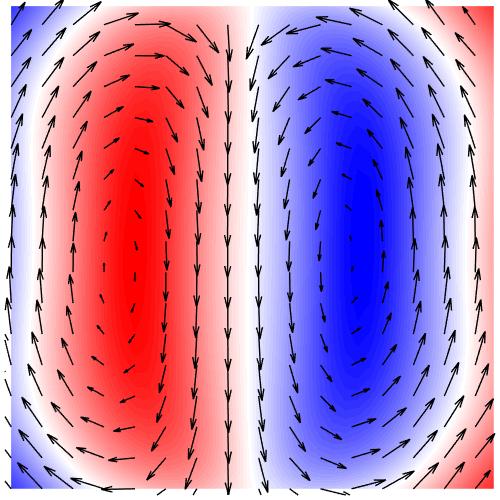}
    \put(5,80){\fcolorbox{black}{white}{\tiny{$1.5\;\mathrm{ns}$}}}
  \end{overpic}
  \begin{overpic}
    [width=.8in]{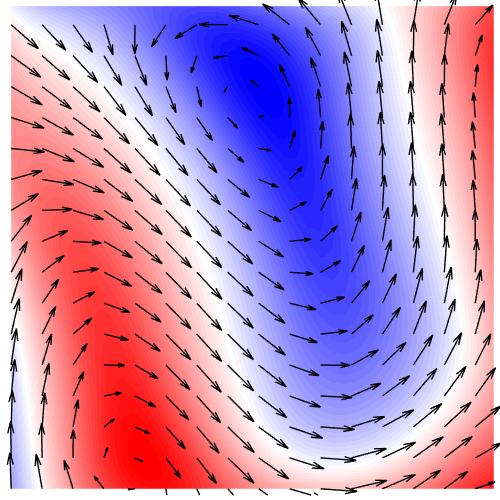}
    \put(5,80){\fcolorbox{black}{white}{\tiny{$2\;\mathrm{ns}$}}}
  \end{overpic}
  \begin{overpic}
    [width=.8in]{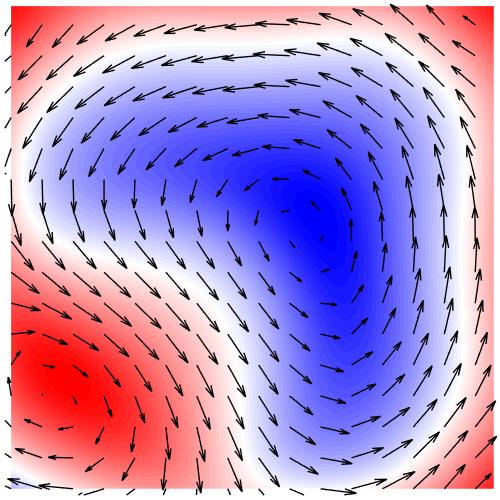}
    \put(5,80){\fcolorbox{black}{white}{\tiny{$2.5\;\mathrm{ns}$}}}
  \end{overpic}
  \begin{overpic}
    [width=.8in]{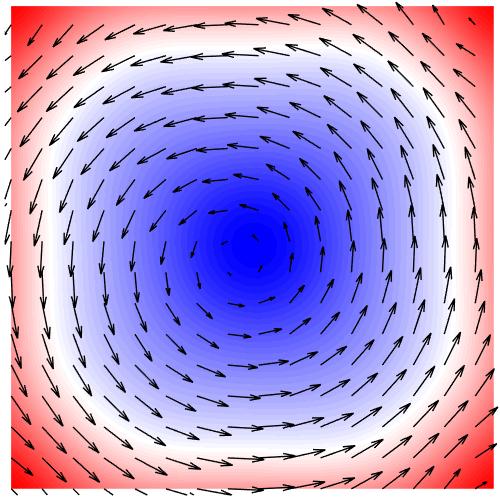}
    \put(5,80){\fcolorbox{black}{white}{\tiny{$3\;\mathrm{ns}$}}}
  \end{overpic}
  \caption{Energy evolution when a spin-polarized current is applied and 7 representative snapshots of magnetization configuration are visualized. Top row: energy evolution driven by the LL equation. Black real lines represent the dynamics when the current is removed, while red and blue dashed lines represent the dynamics when the current is applied with different directions and strengthes. Bottom row: 7 snapshots at different times, corresponding to the pentagrams during the energy evolution.}
\label{fig:80nm_simulations_STT}
\end{figure}

This study then proceeds to investigate the transition between skyrmion clusters with a change in topological number. The transition process is illustrated in \cref{fig:skycluster_7to6}, where a skyrmion escapes through a boundary, leading to a transition between two skyrmion clusters and a subsequent reduction in the system's topological number. It is worth noting that although there are alternative mechanisms for inducing this transition, such as the collapse of a skyrmion or the merger of two skyrmions into one, the transition path illustrated in \cref{fig:skycluster_7to6} represents the MEP in this particular case.
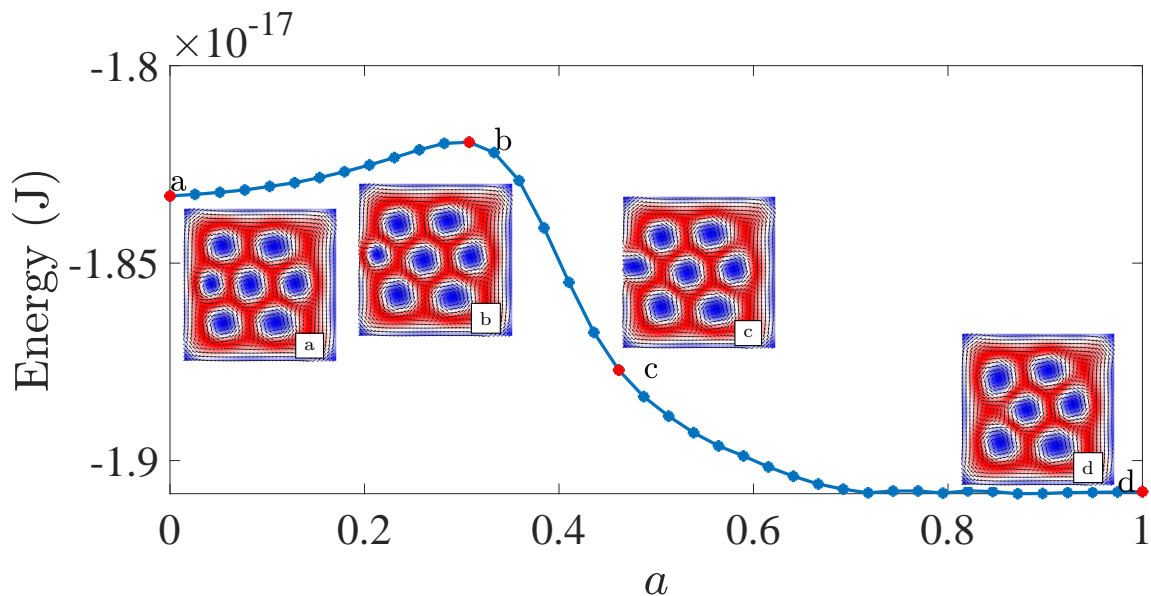
\begin{figure}[htbp]
  \centering
  \begin{overpic}
    [width=6.5in]{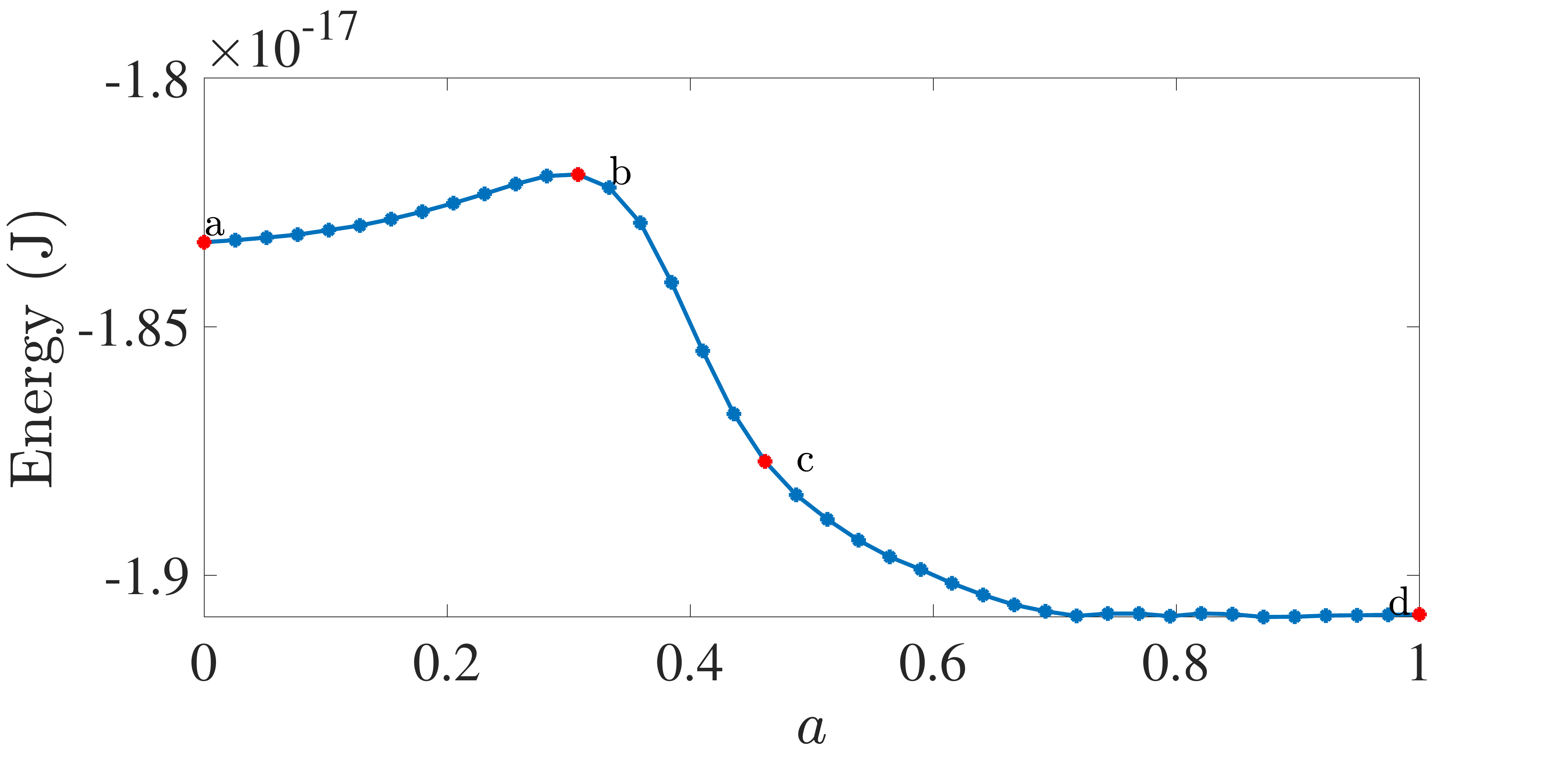}
\put(14,20){\includegraphics[width=0.8in]{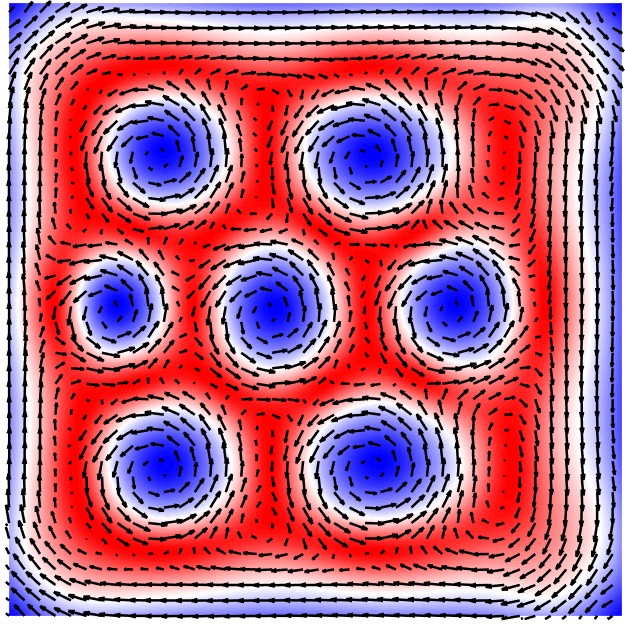}}
    \put(23,21){\fcolorbox{black}{white}{\tiny{a}}}
\put(28,22){\includegraphics[width=0.8in]{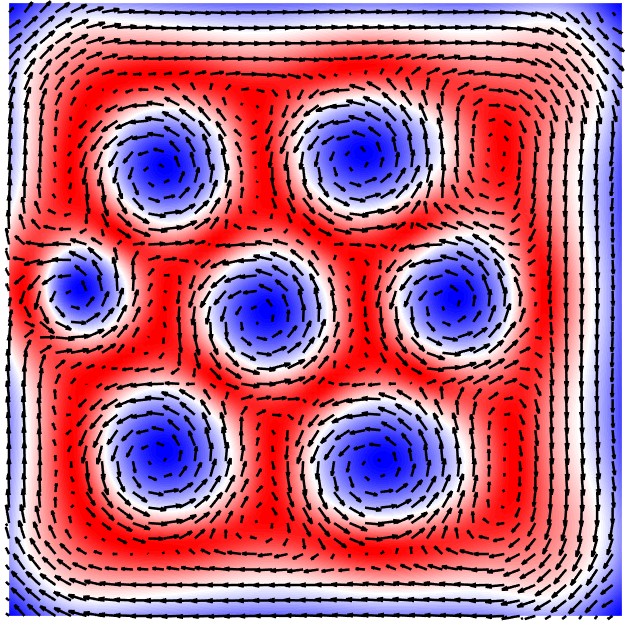}}
    \put(37,23){\fcolorbox{black}{white}{\tiny{b}}}
\put(49,21){\includegraphics[width=0.8in]{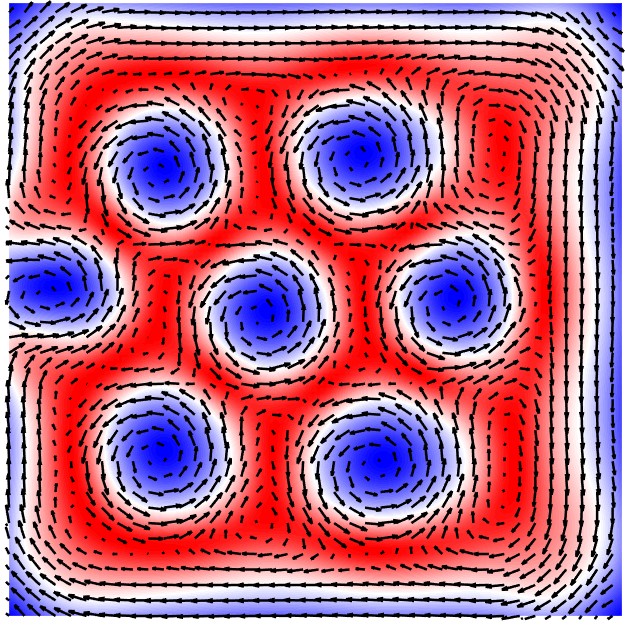}}
    \put(58,22){\fcolorbox{black}{white}{\tiny{c}}}
\put(76,10){\includegraphics[width=0.8in]{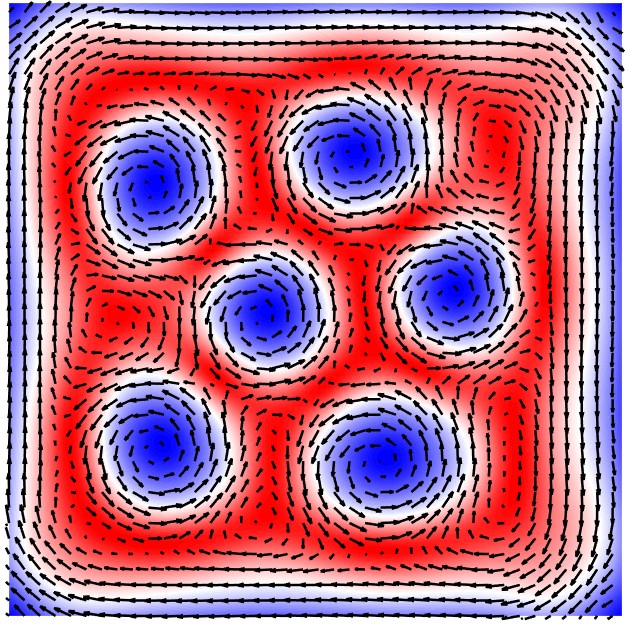}}
    \put(85,11){\fcolorbox{black}{white}{\tiny{d}}}
  \end{overpic}\\
  \caption{Phase transition between skyrmion clusters and a skyrmion escapes from a boundary.}
\label{fig:skycluster_7to6}
\end{figure}

The uniform application of an external field (i.e., magnetic or current field) leads to the simultaneous movement of all skyrmions within the cluster. Hence, it becomes challenging to induce transitions between skyrmion clusters using this approach. As a viable alternative, we propose the use of a local magnetic field to manipulate individual skyrmions. To demonstrate this, we consider the use of a magnetic field to pull the central skyrmion from the cluster, causing it to eventually escape through a boundary. In this demonstration, the magnetic field strength is chosen to be $-2.5\mathbf{e}_{3}\;\mathrm{T}$, which proves adequate for overcoming boundary stiffness, but is considered too strong for inducing desired skyrmion movements.
\begin{figure}[ht]
  \centering
  \begin{overpic}
    [width=1.2in]{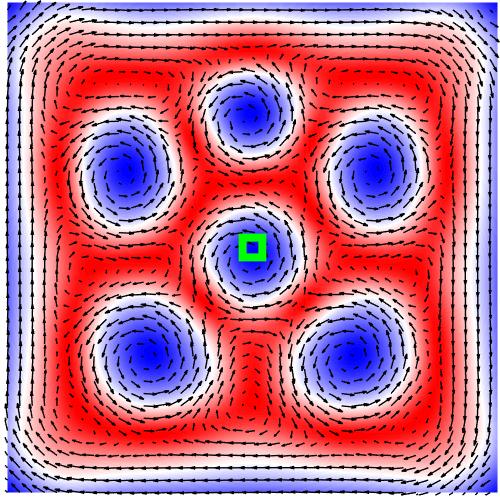}
    \put(5,80){\fcolorbox{black}{white}{\tiny{$0\;\mathrm{ps}$}}}
  \end{overpic}
  \begin{overpic}
    [width=1.2in]{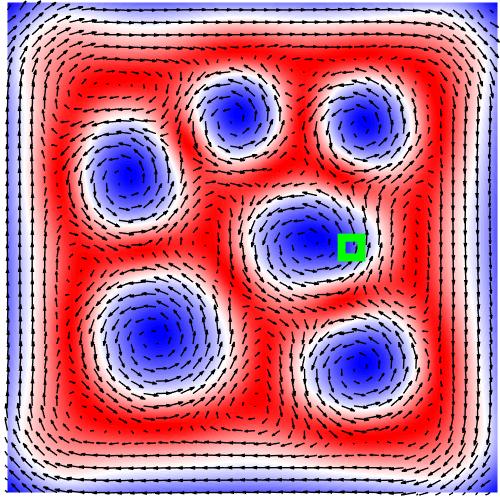}
    \put(5,80){\fcolorbox{black}{white}{\tiny{$1\;\mathrm{ns}$}}}
  \end{overpic}
  \begin{overpic}
    [width=1.2in]{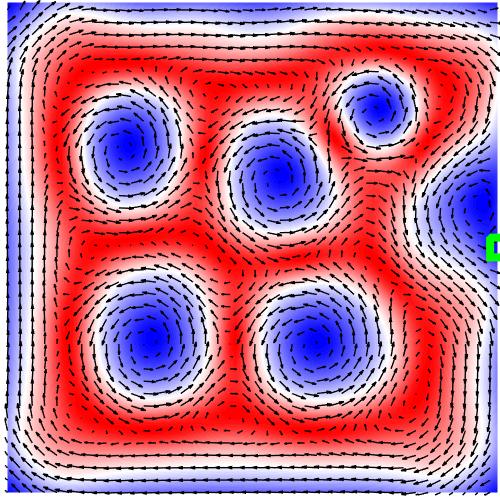}
    \put(5,80){\fcolorbox{black}{white}{\tiny{$2.5\;\mathrm{ns}$}}}
  \end{overpic}
  \begin{overpic}
    [width=1.2in]{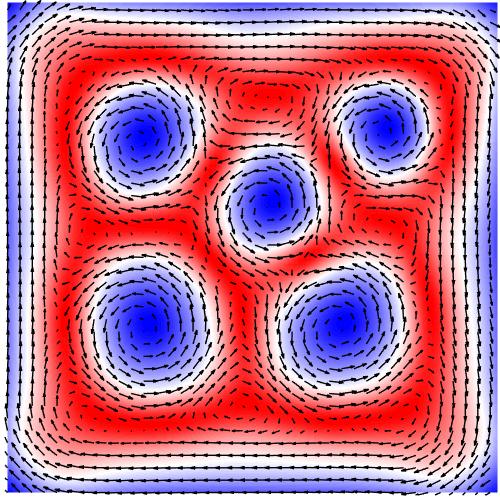}
    \put(5,80){\fcolorbox{black}{white}{\tiny{$3\;\mathrm{ns}$}}}
  \end{overpic}
  \begin{overpic}
    [width=1.2in]{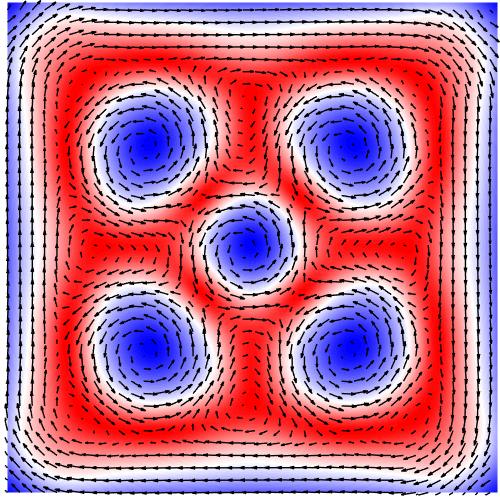}
    \put(5,80){\fcolorbox{black}{white}{\tiny{$12.5\;\mathrm{ns}$}}}
  \end{overpic}
  \caption{A local out-of-plane magnetic field applied over the green square domain drives one skyrmion across the boundary. The magnetic field with magnitude $-2.5\mathbf{e}_{\mathrm{3}}\;\mathrm{T}$ moves along the $\mathbf{e}_{\mathrm{1}}$ direction with velocity $40\;\mathrm{m}/\mathrm{s}$.}
\end{figure}
In addition, skyrmions can be attracted or repelled by an out-of-plane magnetic field. For instance, when a local field is applied over an in-plane domain of $40\times40\;\mathrm{nm}^2$ during a time interval of $0\sim 200\;\mathrm{ps}$, two neighboring skyrmions are observed to move towards each other, eventually resulting in their merger.
\begin{figure}[ht]
  \centering
  \begin{tikzpicture}
    %\draw[help lines,step = 1] (0,0) grid (6,6);
    \node[anchor=south west,inner sep=0] at (0,0) {\begin{overpic}
    [width=1.2in]{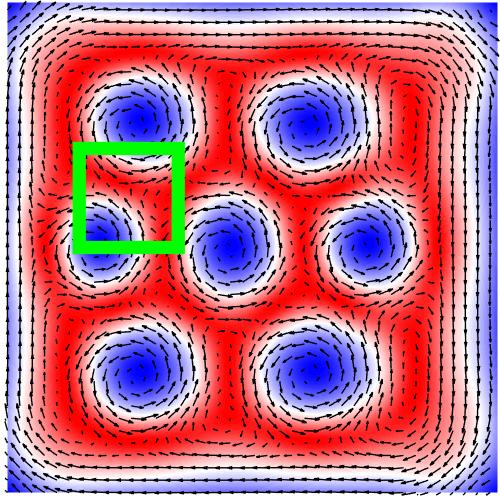}
    \put(5,8){\fcolorbox{black}{white}{\tiny{$0\;\mathrm{ps}$}}}
  \end{overpic}};
\end{tikzpicture}
\begin{tikzpicture}
    \node[anchor=south west,inner sep=0] at (0,0) {\begin{overpic}
    [width=1.2in]{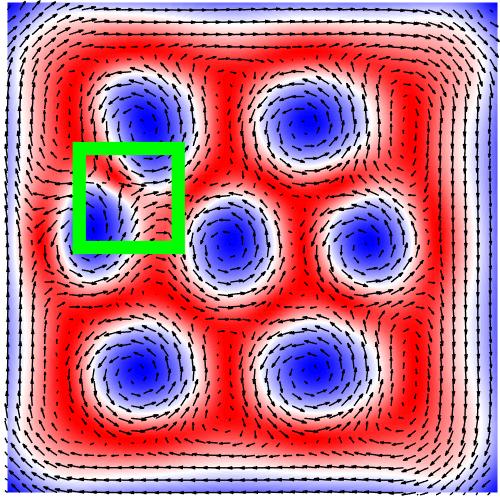}
    \put(5,8){\fcolorbox{black}{white}{\tiny{$40\;\mathrm{ps}$}}}
  \end{overpic}};
\end{tikzpicture}
\begin{tikzpicture}
    \node[anchor=south west,inner sep=0] at (0,0) {\begin{overpic}
    [width=1.2in]{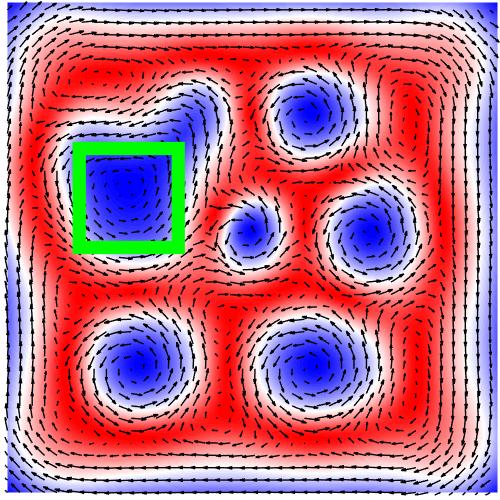}
    \put(5,8){\fcolorbox{black}{white}{\tiny{$200\;\mathrm{ps}$}}}
  \end{overpic}};
\end{tikzpicture}
\begin{tikzpicture}
    \node[anchor=south west,inner sep=0] at (0,0) {\begin{overpic}
    [width=1.2in]{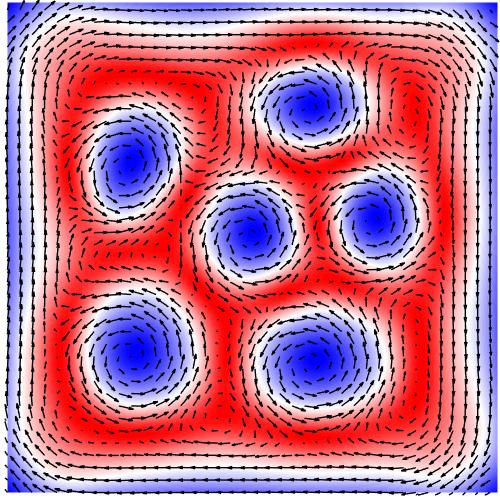}
    \put(5,8){\fcolorbox{black}{white}{\tiny{$400\;\mathrm{ps}$}}}
  \end{overpic}};
\end{tikzpicture}
%\begin{tikzpicture}
  \begin{overpic}
    [width=1.2in]{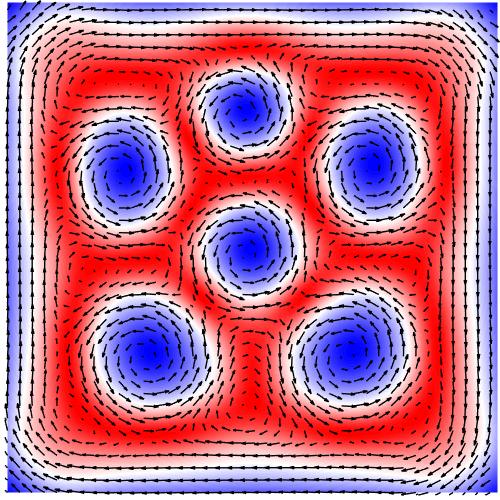}
    \put(5,8){\fcolorbox{black}{white}{\tiny{$5.8\;\mathrm{ns}$}}}
  \end{overpic}
  \caption{Mergence of two skyrmions when an out-of-plane magnetic field $-0.5\mathbf{e}_{\mathrm{z}}\;\mathrm{T}$ is applied over the green square domain within the time period $[0, 200\;\mathrm{ps}]$. }
  \label{fig:snopshots6to5}
\end{figure}
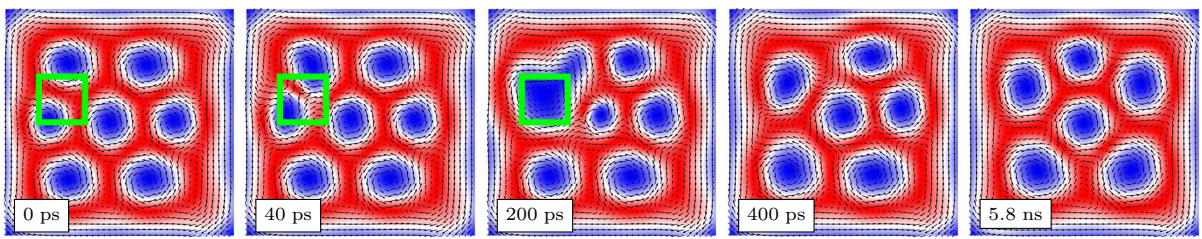

\section{Conclusion}
\label{sec:conclusion}

This study proposes a generalized, second-order accurate, semi-implicit projection scheme for solving the Landau-Lifshitz (LL) equation with the Dzyaloshinskii-Moriya interaction (DMI), which enables the use of larger step-sizes for micromagnetics simulations. It is observed that the LL system exhibits a dynamic instability, and that various stable magnetization configurations can be generated by means of simple initialization as the damping parameter varies, including isolated skyrmions, isolated skyrmionium, skyrmion clusters, skyrmionium clusters, and combinations thereof, in a controlled manner. The string method is employed to identify minimal energy paths connecting different stable magnetization configurations. In particular, the transition between a skyrmion with $Q = 1$ and one with $Q = -1$ involves a local minimizer characterized by a skyrmion junction with $Q = 0$. Moreover, for skyrmion clusters, a transition path is determined that involves a skyrmion escaping through the boundary. The proposed method offers a dependable strategy for studying skyrmion textures and their transition paths, which can greatly enhance our understanding of magnetization dynamics for spintronics applications.

\section*{Acknowledgments}
P. Li thanks for the helpful discussion of Zhiwei Sun, and acknowledges the program of China Scholarships Council No. 202106920036. S. Gu acknowledges the support of NSFC 11901211 and the Natural Science Foundation of Top Talent of SZTU GDRC202137. J. Lan acknowledges the support of NSFC (Grant No. 11904260) and Natural Science Foundation of Tianjin (Grant No. 20JCQNJC02020). J. Chen acknowledges the support of NSFC (Grant No. 11971021). R. Du was supported by NSFC (Grant No. 12271360).

%%-------------------------------------------------------------------------------------------
\bibliographystyle{unsrt}%{amsplain}
\bibliography{refs}

\end{document}